\documentclass[11pt]{amsart}
\usepackage{latexsym,amsfonts,amssymb,amsthm,amsmath,mathrsfs,color,amscd,graphicx,fullpage,parskip,hyperref,bm,bbm,dsfont}
\usepackage[T1]{fontenc}

\usepackage{comment}
\includecomment{uncomment}
\usepackage{commath,mathtools,setspace}
\usepackage{paralist}
\usepackage{extarrows}

\excludecomment{comment} 

\newcommand{\s}[1]{{\mathcal #1}}
\newcommand{\sr}[1]{{\mathscr #1}}
\newcommand{\bb}[1]{{\mathbb #1}}
\newcommand{\vd}[2]{\dfrac{\delta #1}{\delta #2}}

\newcommand{\ip}[2]{\left\langle #1,#2 \right\rangle}

\newtheorem{theorem}{Theorem} 

\newtheorem{lemma}[theorem]{Lemma}
\newtheorem{proposition}[theorem]{Proposition}

\newtheorem{remark}[theorem]{Remark}
\newtheorem{assumption}[theorem]{Assumption}

\numberwithin{equation}{section}
\numberwithin{theorem}{section}

\newcounter{step}
\begin{document}
	
	\title
	{Remarks on potential mean field games}
	
	\author{P.~Jameson Graber}
	\thanks{The author is grateful to be supported by National Science Foundation through NSF Grant DMS-2045027.}
	\address{J. Graber: Baylor University, Department of Mathematics;\\
		Sid Richardson Building\\
		1410 S.~4th Street\\
		Waco, TX 76706\\
		Tel.: +1-254-710- \\
		Fax: +1-254-710-3569 
	}
	\email{Jameson\_Graber@baylor.edu}
	
	
	\date{\today}   
	
	\begin{abstract}
		In this expository article, we give an overview of the concept of potential mean field games of first order.
		We give a new proof that minimizers of the potential are equilibria by using a Lagrangian formulation.
		We also provide criteria to determine whether or not a game has a potential.
		Finally, we discuss in some depth the selection problem in mean field games, which consists in choosing one out of multiple Nash equilibria.
	\end{abstract}
	

	\maketitle
	
	
	\section{Introduction}

	Mean field game theory was introduced around 20 years ago as a framework for analyzing the behavior of large numbers of agents playing strategically, with the earliest references generally attributed to Lasry and Lions \cite{lasry06,lasry06a,lasry07} and Caines, Huang, and Malham\'e \cite{huang2006large}.
	Since then there has been a wealth of literature published on both mean field games and its companion theory of mean field type control, which can be defined as the control of McKean-Vlasov type equations or, alternatively, of transport type partial differential equations (PDE).
	The interested reader can examine the lecture notes \cite{achdou2021mean}, the monographs \cite{carmona2017probabilistic,carmona2017probabilisticII,bensoussan2013mean}, the lectures of Lions at the Coll\`ege de France \cite{lions07} (see the notes by Cardaliaguet \cite{cardaliaguet2010notes}), and the survey article \cite{gomes2014mean}.
	In this article, we are going to focus on the case in which the two theories overlap, i.e.~a mean field game that is also a mean field type control problem; such games are called \emph{potential games}.
	It is intended to be expository.
	Though we will prove one new result, the emphasis will be on explaining what is known about the topic using a wide variety of sources.
	It is nevertheless inevitable that many relevant sources will be missing, and for this the author apologizes in advance.
	
	Let us briefly state the idea of a mean field game.	
	Assume there is a large number of agents, each of which chooses a trajectory $\gamma$ in order to minimize some cost, which may have the form
	\begin{equation} \label{cost}
		J(m,\gamma) = \int_0^T \del{L\del{\gamma(t),\dot{\gamma}(t)} + f\del{m(t),\gamma(t)}}\dif t + g\del{m(T),\gamma(T)}.
	\end{equation}
	Here $m(t)$ represents a distribution of states (or positions) at time $t$.
	The idea is that any given player will have knowledge of the distribution of the states of all agents, at least through the functions $f$ and $g$.
	For any initial state $x$, we say that a curve $\gamma_x(t)$ is a \emph{best response} to the flow $m(t)$ provided that $\gamma = \gamma_x(\cdot)$ minimizes $J(m,\gamma)$.
	Now fix an initial measure $m_0$, representing the distribution of the players' starting positions.
	If every player chooses a best response to the flow $m(t)$, then there is a resulting flow of measure $\mu(t)$, giving the distribution of players' positions at time $t$ (for instance, if the best response $\gamma_x(t)$ is unique for each $x$, then $\mu(t)$ is the push-forward of $m_0$ through the map $x \mapsto \gamma_x(t)$).
	We call $\mu(t)$ a best response to $m(t)$.
	A \emph{Nash equilibrium} is defined to be a flow of measures $m(t)$ that is a best response to itself.
	
	Using principles from optimal control theory and continuum mechanics, we can derive a system of PDE to characterize Nash equilibrium.
	Introducing a value function
	\begin{equation}
		u(t,x) = \inf_{\gamma(t) = x}\cbr{\int_t^T \del{L\del{\gamma(\tau),\dot{\gamma}(\tau)} + f\del{m(\tau),\gamma(\tau)}}\dif t + g\del{m(T),\gamma(T)}},
	\end{equation}
	we see that formally $u$ should be the solution of a Hamilton-Jacobi equation
	\begin{equation*}
		-\partial_t u(t,x) + H\del{x,D_x u(t,x)} = f\del{m(t),x}
	\end{equation*}
	with Hamiltonian given by taking the Legendre transform of $L$ with respect to the velocity:
	\begin{equation}
		H(x,p) := L^*(x,-p) = \sup_v \del{-p \cdot v - L(x,v)}.
	\end{equation}
	Then the optimal feedback control is (formally) given by $-D_p H\del{x,D_x u(x,t)}$, i.e.~best response curves are trajectories of the flow
	\begin{equation*}
		\dot{\gamma}(t) = -D_p H\del{\gamma(t),D_x u\del{\gamma(t),t}}.
	\end{equation*}
	From continuum mechanics, this implies that the best response $\mu(t)$ solves the continuity equation
	\begin{equation*}
		\partial_t \mu - \nabla \cdot \del{D_p H(x,D_x u)\mu} = 0.
	\end{equation*}
	In Nash equilibrium, $\mu = m$.
	The resulting system of PDE for the pair $(u,m)$ is
	\begin{equation} \label{mfg system}
		\begin{aligned}
			&-\partial_t u + H(x,D_x u) = f(m,x), &x \in \Omega, t \in (0,T),\\
			&\partial_t m - \nabla \cdot \del{D_p H(x,D_x u)m} = 0, &x \in \Omega, t \in (0,T)\\
			&m(0) = m_0, \quad u(x,T) = g(m(T),x).
		\end{aligned}
	\end{equation}
	The spatial domain $\Omega$ could be all of Euclidean space $\bb{R}^d$, the torus $\bb{T}^d := \bb{R}^d / \bb{Z}^d$ in the case of periodic boundary conditions, or a bounded domain with smooth boundary, on which we impose a no-flux boundary condition, formally written ${\bf n} \cdot D_p H(x,D_x u)m = 0$ for $x \in \partial \Omega$.
	Here and in the entire article we treat only \emph{first order} mean field games, meaning that the control problem is purely deterministic, and thus the resulting PDE system is of first order.
	
	In \cite[Section 2.6]{lasry07} Lasry and Lions make the observation that \eqref{mfg system} has an interpretation as a first-order condition of optimality for the optimal control of the flow of measures.
	Formally, solutions of \eqref{mfg system} can be derived from finding critical points of the following cost functional:
	\begin{equation} \label{potential}
		\s{J}(m,v) = \int_0^T \del{\int_{\Omega} L\del{x,v(t,x)}m(t,\dif x) + F\del{m(t)}}\dif t + G\del{m(T)}.
	\end{equation}
	Here $F$ and $G$ are \emph{potentials} of $f$ and $g$, respectively, in the sense that $f$ (resp.~$g$) is a linear derivative of $F$ (resp.~$G$), as defined in Section \ref{sec:preliminaries} below.
	The argument $v = v(t,x)$ is the vector field that generates the flow of measures, i.e.~$m$ and $v$ are related through the continuity equation
	\begin{equation} \label{continuity eq}
		\partial_t m + \nabla \cdot (mv) = 0,
	\end{equation}
	which holds in the sense of distributions.
	The cost appearing in \eqref{potential} can be thought of as a \emph{potential} of the individual cost appearing in \eqref{cost}, i.e.~if we take the derivative of \eqref{potential} with respect to the measure $m$, then we will get \eqref{cost}.
	We will make this claim more precise below in Sections \ref{sec:preliminaries} and \ref{sec:lagrangian}.
	
	Lasry and Lions' observation about the mean field game proposed in \cite{lasry07} implies that it is a \emph{potential game}, a term which has long been used in the literature on game theory \cite{monderer1996potential}.
	In the mathematics literature on mean field games, one does not always see the term ``potential game'' explicitly used.
	However, many of the first results on weak solutions for first-order models, or second-order models with degenerate diffusion, were based on the fact that those games are potential \cite{cardaliaguet2015weak,cardaliaguet2014mean,cardaliaguet2015second,cardaliaguet2016first}.
	The methods in these papers are largely inspired by optimal transport theory, especially the Benamou-Brenier dynamic formulation from \cite{benamou2000computational} (see also \cite{cardaliaguet2013geodesics}).
	We will outline these methods below in Section \ref{sec:pde optimization}.
	
	The purpose of this article is to provide an overview of known results and open problems about potential mean field games.
	It will be useful to draw on a wide range of sources, ranging from PDE literature to theoretical economics.
	There are at least two reasons to take a special interest in potential games.
	One is computational: there is a rich collection of mathematical techniques that can be used to find minimizers of a potential, and if we can prove such a minimizer exists, then we have proved the existence of a Nash equilibrium as well.
	Another reason to study potential games is out of consideration for the \emph{selection problem}, which we can briefly summarize as follows: if there are multiple Nash equilibria, how can we rationally choose one as preferable to another?
	In a potential game, all minimizers of the potential are Nash equilibria, but not all equilibria need to be minimizers of the potential.
	Thus, to give a preview of things to come, we can say the potential provides a \emph{selection principle}, i.e.~a justification for choosing one Nash equilibrium over others.
	Selection will be a major theme of this article and will be covered in Section \ref{sec:selection}.
	
	In the remainder of this introduction, we will lay out the history of potential games from two different points of view.
	First, we will explain how the literature on mean field games has taken the perspective of PDE constrained optimization in order to produce both theoretical and computational results.
	Then we will take a brief look at some economics literature on potential games and see how it relates to our mean field setting.
	Although the results from this literature cannot be immediately applied to differential games, they are nevertheless illuminating and suggest certain avenues of future research in mean field games.
	Following this historical background, we will shift points of view from the Eulerian (i.e.~using the PDE system) to the Lagrangian (tracking the distribution of paths), which will provide the setting for this paper's main result.

	\subsection{PDE constrained optimization} \label{sec:pde optimization}
	Assume the potential $\s{J}$ from \eqref{potential} is convex in the variables $(m,mv)$.
	One can construct a unique Nash equilibrium using the following steps (cf.~\cite{benamou2017variational}).
	\begin{enumerate}
		\item The constraint \eqref{continuity eq}, i.e.~the continuity equation, can be written in weak form as follows:
		\begin{equation*}
			\int_0^T \int_{\Omega} \del{\partial_t u(t,x) + D_x u(t,x)\cdot v(t,x)}m(t,\dif x)
			- \int_{\Omega} u(T,x)m(T,\dif x) + \int_{\Omega} u(0,x)m_0(\dif x) = 0
		\end{equation*}
		for all smooth test functions $u$.
		Using $u$ as a Lagrange multiplier, we can rewrite the constrained problem in Lagrangian form as
		\begin{align*}
			\inf_{(m,v)} \s{J}(m,v) &= \inf_{(m,v)} \sup_u \s{J}(m,v) + \int_0^T \int_{\Omega} \del{\partial_t u(t,x) + D_x u(t,x)\cdot v(t,x)}m(t,\dif x)\\
			&\quad - \int_{\Omega} u(T,x)m(T,\dif x) + \int_{\Omega} u(0,x)m_0(\dif x).
		\end{align*}
		\item Use a formal min-max exchange to identify a dual problem $\s{J}^*$ for the adjoint state $u$:
		\begin{align*}
			\inf_{(m,v)} \s{J}(m,v) 
			&= \sup_u \inf_{(m,v)} \int_0^T \del{\int_{\Omega} \del{L\del{v(t,x)} + \partial_t u(t,x) + D_x u(t,x)\cdot v(t,x)}m(t,\dif x) + F\del{m(t)}}\dif t\\
			&\quad 	+ G\del{m(T)} - \int_{\Omega} u(T,x)m(T,\dif x) + \int_{\Omega} u(0,x)m_0(\dif x)\\
			&= \sup_u \inf_{m} \int_0^T \del{\int_{\Omega} \del{\partial_t u(t,x) - H\del{x,D_x u(t,x)}}m(t,\dif x) + F\del{m(t)}}\dif t\\
			&\quad 	+ G\del{m(T)} - \int_{\Omega} u(T,x)m(T,\dif x) + \int_{\Omega} u(0,x)m_0(\dif x)\\
			&= \sup_u \int_0^T -F^*\del{-\partial_t u(t,\cdot) + H\del{\cdot,D_x u(t,\cdot)}}\dif t - G^*\del{u(T,\cdot)}
			=: -\inf_u \s{J}^*(u).
		\end{align*}
		\item Using the Fenchel-Rockafellar Theorem \cite{ekeland1976convex}, one can establish the above identity rigorously, and at the same time one proves that $\s{J}$ has a minimizer $(m,v)$.
		\item To prove the existence of an adjoint state, one must suitably relax the dual problem $\s{J}^*(u)$.
		With a slight abuse of notation, we define
		\begin{equation*}
			\s{J}^*(u,\alpha) = \int_0^T F^*\del{\alpha(t,\cdot)}\dif t + G^*\del{u(T,\cdot)}
		\end{equation*}
		for pairs $(u,\alpha)$ satisfying
		\begin{equation*}
			-\partial_t u + H\del{x,D_x u} \leq \alpha.
		\end{equation*}
		Under suitable conditions on $F$ and $G$, it is possible to prove the existence of a minimizing pair $(u,\alpha)$.
		Another technical step is to prove that $\min \s{J}^*(u,\alpha) = \inf \s{J}^*(u)$.
		See for instance \cite[Section 2.2]{cardaliaguet2014mean}.
		See also \cite{graber2014optimal}.
		\item Finally, one can show that the minimizers of primal and dual problems satisfy the PDE system \eqref{mfg system} in some suitable weak sense.
		For first-order and degenerate second-order problems in which $m$ is required to be a density (and $F$ and $G$ depend on pointwise values of that density), one does not necessarily have that $u$ satisfies the Hamilton-Jacobi equation in the viscosity sense or even in the sense of distributions; this is replaced by an even weaker criterion that essentially says $u$ behaves like the value function on the support of $m$.
		See \cite{cardaliaguet2014mean,cardaliaguet2015second} for a precise definition of weak solutions.
	\end{enumerate}
	It is not necessary to have convexity (or uniqueness of minimizers) in order to prove that minimizers of the potential are equilibria.
	If one takes any minimizer $(\bar{m},\bar{v})$ of $\s{J}$, one can show that it also minimizes the same functional with $F(m)$ and $G(m)$ replaced by linear approximations near $\bar{m}$; the rest of the argument proceeds by duality as above.
	See \cite[Proposition 3.1]{briani2018stable} for details.
	
	In addition to proving the existence of equilibria, the above outline also motivates numerical algorithms to compute solutions, since it implies that a solution is equivalent to a saddle point for the Lagrangian
	\begin{equation*}
		\sr{L}(m,v,u) := \s{J}(m,v) + \s{J}^*(u) + \int_{\Omega}u(0)m_0 - \int_{\Omega} u(T)m(T) + \int_0^T \int_{\Omega} \del{\partial_t u + D_x u \cdot v}m.
	\end{equation*}
	See, for instance, \cite{benamou2015augmented,benamou2017variational}.
	This method hearkens back to Benamou and Brenier's computational approach to the Monge-Kantorovitch problem in optimal transportation, which they reformulate as the control of a continuity equation \cite{benamou2000computational}.

	\subsection{Connection to the theoretical economics literature}
	
	The concept of a potential game goes back a long way in the literature on game theory and economics; see, for instance, the overview given by Monderer and Shapley \cite{monderer1996potential}.
	For finite games, a potential game is one for which a change in any player's cost function is matched precisely by a corresponding change in the potential function.
	To be precise, suppose player $i \in \{1,\ldots,N\}$ choose a strategy $a_i$ from a finite set $S_i$, and for a vector $a = (a_1,\ldots,a_n)$ denote by $f_i(a) = f_i(a_i,a_{-i})$ the cost to player $i$ (the notation $(a_i,a_{-i})$ is standard in game theory for singling out the component $a_i$, with $a_{-i}$ representing all the other components).
	(NB: in the economics literature it is more common to think of maximizing utility instead of minimizing cost; as mathematicians, it seems, we insist on multiplying all functions by $-1$.)
	A Nash equilibrium is a vector $a$ such that
	\begin{equation*}
		f_i(\tilde{a}_i,a_{-i}) \geq f_i(a_i,a_{-i}) \quad \forall \tilde{a}_i, \ \forall i \in \{1,\ldots,N\}.
	\end{equation*}
	The function $P$ is a potential for the game provided
	\begin{equation*}
		f_i(\tilde a_i,a_{-i}) - f_i(a_i,a_{-i}) \equiv P(\tilde a_i,a_{-i}) - P(a_i,a_{-i})
	\end{equation*}	
	for each $i$.
	In this setting, it is immediate that Nash equilibria correspond to local minimizers of $P$.
	Indeed, if $a$ is a minimizer of $P$, then for player $i$ to change strategies from $a_i$ to $\tilde{a}_i$ can only mean an increase in $P$, which corresponds to an increase in $f_i$; since this is true for every $i$, $a$ is an equilibrium.
		
	The concept of a potential game seems to have been first introduced in \cite{rosenthal1973class}, which defined a particular class known as \emph{congestion games}.
	Roughly speaking, a congestion game is one for which each player's utility depends on the total number of players sharing the same spaces (e.g.~road segments) that appear in the player's strategy.
	It can be proved that every finite congestion game has a potential and that any minimizer of the potential is also a Nash equilibrium.
	In \cite{monderer1996potential} the authors show that every finite potential game is isomorphic to a congestion game.
	Heuristically, one might posit that the equivalent of congestion games for mean field games is one in which $f$ and $g$ depend on $m$ through its density, i.e.~$f(m,x) = f(\rho(x),x)$ whenever $\dif m(x) = \rho(x)\dif x$.
	It is a remarkable coincidence that much of the early literature on the PDE system \eqref{mfg system} deals with precisely this case, or with a related case in which $f(m,x) = \int \xi(x-y)f\del{\int \xi(y-z)\dif m(z),y}\dif y$ where $\xi$ is a standard convolution kernel.
	One may even be tempted to conjecture that all potential mean field games are in some sense isomorphic to an appropriately defined ``congestion game;'' however, this is beyond the scope of the present article.
	
	On the other hand, ``congestion'' in mean field games may also refer to the case where the Lagrangian itself depends in a nonlinear way on the density $m$ in such a manner as to further penalize players when they traverse regions of high density at high speeds.
	This has been studied, e.g.~in \cite{achdou2016mean,achdou2016meanII,achdou2018mean,gomes2015short,graber2015weak}; cf.~\cite{cardaliaguet2013geodesics}.
	However, note that in this case, the mean field game is \emph{not} potential, which we will formally prove in Section \ref{sec:whenispotential}.
	Thus the conjecture from the previous paragraph must not be taken too na\"ively, as it is clearly false if interpreted too broadly.
	
	Monderer and Shapley also prove that potential games have the \emph{fictitious play} property \cite{monderer1996fictitious}.
	This means players can progressively ``learn'' a Nash equilibrium through an iterative process in which the next strategy chosen will be the best response to the average of all previous distributions of strategies \cite{brown1951iterative}.
	If every limit point of this sequence is a Nash equilibrium, the game is said to have the fictitious play property.
	Such a property provides additional justification for the concept of Nash equilibrium: in addition to players having rationally consistent strategies, there is a plausible path through which these strategies can gradually evolve.
	The fictitious play property has also been proved for potential mean field games by Cardaliaguet and Hadikhanloo \cite{cardaliaguet2017learning}; see also \cite{briani2018stable}.
	More general, Sandholm has observed that for a potential game, there exists a general criterion for population dynamics that lead to Nash equilibrium, namely the familiar concept of Lyapunov stability \cite{sandholm2001potential}.
	That is, if the potential is a Lyapunov function for a dynamical system on the space of population distributions, then one can deduce that through such dynamics the population will converge to a rest point; if the dynamics are such that rest points are all Nash equilibria, then we know that each equilibrium is stable, hence we expect them to persist.
	To state and prove similar results for mean field games would be considerably more technical, since the space of population distributions in this case would be probability measures on the space of continuous curves.
	We note, however, that particular evolutionary games have already been studied in a mean field setting; see for instance \cite{ambrosio2021spatially} and references therein.
	We will not pursue this direction in the present paper, but it would make for interesting future research.
	
	An analog of finite potential games for a continuum of players has been given by Sandholm \cite{sandholm2001potential}.
	(This does not quite constitute a mean field game because it is not a differential game; however, the spirit is the same, in the sense that one can view it as the limit of $N$-player games as $N \to \infty$.)
	Let us give the simplest version of such a game.
	Assume there are $n$ strategies, labeled $1,\ldots,n$.
	Instead of individual strategies, we will consider \emph{distributions} of strategies, which are vectors in the simplex $\Sigma = \{x \in \bb{R}^n : x_j \geq 0, x_1 + \cdots + x_n = 1\}$.
	For $x \in \Sigma$, the \emph{support} of $x$ is the set of all $j$ such that $x_j > 0$, i.e.~the set of strategies that are being used by someone.
	The cost function for the $j$th strategy is denoted $f_j(x)$.
	We say that $a \in \Sigma$ is a Nash equilibrium provided that, for every $j$ in the support of $a$, we have $f_j(a) \leq f_i(a)$ for all $i$.
	In other words, if there are any players using strategy $j$, it is because $j$ is an optimal strategy to use given the population distribution $a$.
	We say the game has a potential if there is some $F:\bb{R}^n \to \bb{R}$ for which $f = (f_1,\ldots,f_n) = \nabla F$.
	
	In this setting, it is straightforward, though not quite immediate, to see that minimizers of $F$ are also equilibria.
	Again, we have only to consider what would happen if some people changed their current strategy.
	We have the following simple result.
	\begin{proposition} \label{simple result}
		Suppose $f = (f_1,\ldots,f_n) = \nabla F$ for some potential $F$.
		Then the minimizers of $F$ are also equilibria.
	\end{proposition}

	\begin{proof}
		Let $a$ minimize $F$, let $j$ be such that $a_j > 0$, and let $i \neq j$.
		Define $e_{i}$ to be the $i$th standard basis vector in $\bb{R}^n$.
		For sufficiently small $h > 0$, $a + he_i - he_j \in \Sigma$ and represents the strategy distribution that results from some players switching from $j$ to $i$. 
		By minimality, $F(a + he_i - he_j) \geq F(a)$ for all $h > 0$.
		We deduce $\nabla F(a) \cdot (e_i - e_j) \geq 0$, which is the same as $f_i(a) - f_j(a) \geq 0$, as desired.
	\end{proof}
	\begin{remark}
		Notice that in the proof of Proposition \ref{simple result}, we do not need the full gradient of $F$, but only its derivative in the simplex, i.e.~$\lim_{h \to 0} \frac{F\del{a + h(b-a)} - F(a)}{h}$ for vectors $a,b \in \Sigma$.
		Cf.~\cite{sandholm2009large}.
		This is useful to keep in mind when we discuss derivatives over the space of measures below.
	\end{remark}
	Proposition \ref{simple result} has an immediate corollary: if $F$ has at least one minimizer, then the game has at least one equilibrium in pure strategies.
	The proof has the virtue of being transparent, connecting the switching of strategies directly to the increase or decrease in the potential.
	The central goal of this article is to provide an equally transparent account of potential mean field (differential) games.
	To do this, we will change point of view from Eulerian (i.e.~using the PDE system \eqref{mfg system}) to Lagrangian, in which we view the equilibrium as a measure over the space of curves \cite{benamou2017variational,bonnans2023lagrangian,cannarsa2018existence,cardaliaguet2016first,cardaliaguet2017learning,fischer2021asymptotic,mazanti2019minimal}.
	This will allow us to generalize the proof of Proposition \ref{simple result} in a direct way, though the proof is more technical.

	\subsection{Lagrangian formulation}
	
	The cost $J$ given in \eqref{cost} is a function of an individual trajectory $\gamma$ and a path of measures $m = m(t)$.
	Nash equilibrium is then determined from an Eulerian point of view: we first derive the optimal velocity (feedback control) at each point $(t,x)$ and then use that velocity field to determine the evolution of $m$.
	To switch to the Lagrangian point of view, we instead consider a probability measure $\eta$ on the space of continuous curves that stay in $\Omega$.
	We then define the cost as
	\begin{equation} \label{J}
		J(\eta,\gamma) = \int_0^T \del{L\del{\gamma(t),\dot{\gamma}(t)} + f\del{m^\eta(t),\gamma(t)}}\dif t + g\del{m^\eta(T),\gamma(T)},
	\end{equation}
	where $m^\eta(t) := e_t \sharp \eta$ is the push-forward of $\eta$ through the evaluation map $e_t(\gamma) := \gamma(t)$.
	Nash equilibrium is then defined by the fact that $\eta$ is concentrated on curves $\gamma$ that are optimal with respect to the cost $J(\cdot,\eta)$.
	
	The Lagrangian point of view has been employed in several works on mean field games, e.g.~\cite{benamou2017variational,cardaliaguet2017learning,cannarsa2018existence,mazanti2019minimal,fischer2021asymptotic,bonnans2023lagrangian,gomes2016extended}.
	It is especially useful as a first approach to studying mean field games with state constraints, for which deriving the PDE solution can be quite technical \cite{cannarsa2019constrained,cannarsa2021mean}.
	Fischer and Silva used the Lagrangian point of view to prove that $N$-player Nash equilibria converge to mean field equilibria as $N \to \infty$, for a first-order game \cite{fischer2021asymptotic}.
	This may be compared with the result of Sandholm for games with finite strategy sets \cite{sandholm2001potential}.
	
	In this article, the Lagrangian point of view will allow us to generalize the simple proof of Proposition \ref{simple result} showing that minimizers of the potential are also Nash equilibria.
	The potential is defined by
	\begin{equation} \label{sJ}
		\s{J}(\eta) = \int_\Gamma \int_0^T L\del{\gamma(t),\dot{\gamma}(t)} \dif t \dif \eta(\gamma) + \int_0^T F\del{m^\eta(t)}\dif t + G\del{m^\eta(T)},
	\end{equation}
	where $\Gamma$ is the space of admissible curves $\gamma$.
	Indeed, from the Lagrangian point of view it is now straightforward to see why $\s{J}$ is a potential.
	Formally, one can see that the first variation of $\s{J}$ with respect to the measure variable $\eta$ is given by $J(\eta,\gamma)$, provided that $f$ and $g$ are the first variation of $F$ and $G$, respectively.
	(See Section \ref{sec:preliminaries} for precise definitions.)
	The intuition that minimizers are equilibria is that, if one starts with a particular $\gamma$ in the support of $\eta$ and exchanges it with a different curve $\tilde{\gamma}$, the cost can only increase.
	To make this rigorous, we will have to exchange \emph{neighborhoods} and not only individual trajectories; however, this adds only some minor technical details to an otherwise straightforward argument.
	
	The rest of this article is organized as follows.
	In Section \ref{sec:minimizers are eq}, we prove that minimizers of \eqref{sJ} are Nash equilibria for the game given by \eqref{J}.
	To do this, we start by giving some technical definitions, including derivatives on the space of probability measures.
	The main result is proved following the same form as for Proposition \ref{simple result}, while keeping in mind we have a continuous (and indeed infinite dimensional) set of controls.
	After proving the main result, we briefly address two other questions of interest: whether the Lagrangian formulation is equivalent to the original Eulerian formulation, and whether we can identify which mean field games have a potential.
	Then in Section \ref{sec:selection}, we make some remarks on the selection problem, especially as it pertains to potential mean field games.
	We will draw a strong connection between selecting mean field equilibria and solving nonlinear transport equations for which characteristics may cross.
	We will see that the concept of entropy solutions from the theory of nonlinear balance laws can be plausibly used to resolve the selection problem, though not without difficulty.

	\section{Minimizers of a potential are equilibria} \label{sec:minimizers are eq}
	
	\subsection{Preliminaries} \label{sec:preliminaries}
	
	For a separable metric space $X$, denote by $\sr{P}(X)$ the set of all Borel probability measures on $X$.
	The support of a measure $m \in \sr{P}(X)$, denoted $\operatorname{supp} m$, is defined to be the set of all $x \in X$ such that $m\del{B_\varepsilon(x)} > 0$ for every $\varepsilon > 0$, where $B_\varepsilon(x)$ is the ball of radius $\varepsilon$ around $x$ in $X$.
	If $Y$ is another separable metric space and $f:X \to Y$, then for $m \in \sr{P}(X)$ we define the push-forward $f \sharp m \in \sr{P}(Y)$ by $f \sharp m(B) = m\del{f^{-1}(B)}$.
	If $d$ is the metric on $X$ and $p \geq 1$, we denote by $\sr{P}_p(X)$ the set of all $m \in \sr{P}(X)$ such that
	\begin{equation*}
		\int_X d(x,x_0)^p\dif m(x) < \infty
	\end{equation*}
	for some (hence for all) fixed $x_0 \in X$.
	Note that if $X$ is compact, then $\sr{P}_p(X) = \sr{P}(X)$.
	The Wasserstein metric on $\sr{P}_p(X)$ is given by
	\begin{equation}
		W_p(\mu,\nu)^p = \inf \cbr{\int_{X \times X} d(x,y)^p\dif \pi(x,y) : \pi \in \Pi(\mu,\nu)}
	\end{equation}
	where $\Pi(\mu,\nu)$ is the set of all couplings between $\mu$ and $\nu$, i.e.~the set of all $\pi \in \sr{P}(X \times X)$ such that $\pi(B \times X) = \mu(B)$ and $\pi(X \times B) = \nu(B)$ for all Borel sets $B \subset X$.
	It is well-known that an optimal coupling exists, i.e.~the infimum is a minimum \cite[Theorem 4.1]{villani2008optimal}.
	
	For a function $F:\sr{P}_p(\Omega) \to \bb{R}$, we say that $f(m,x)$ is a \emph{linear derivative} of $F$ at $m \in \sr{P}_p(\Omega)$ provided that
	\begin{equation}
		\lim_{t \to 0} \frac{F\del{m + t(\tilde m - m)} - F(m)}{t} = \int_\Omega f(m,x)\dif (\tilde m - m)(x) \quad \forall \tilde m \in \sr{P}_p(\Omega).
	\end{equation}
	In other words, $f(m,x)$ is essentially the first variation of $F$ at $m$, where the variation is taken only within the space of probability measures.
	Observe that $f(m,x)$ is not uniquely defined; if $\tilde f(m,x)$ is another linear derivative, then by considering $\tilde m = \delta_x$ we get
	\begin{equation}
		f(m,x) - \tilde f(m,x) = \int_{\Omega} \del{f(m,\xi) - \tilde f(m,\xi)}\dif m(\xi)
	\end{equation}
	and thus $f(m,\cdot)$ and $\tilde f(m,\cdot)$ differ by a constant, depending on $m$.
	This non-uniqueness will not be of any real consequence in our study.
	
	\subsection{Assumptions} \label{sec:lagrangian}
	
	Let us the consider a Lagrangian formulation of a certain class of mean field games.
	We fix our spatial domain $\Omega$ to be either $\bb{R}^d$, $\bb{T}^d$, or a compact domain that has a $\s{C}^2$ boundary.
	In the last case, we denote by $d_\Omega(x) = \inf\cbr{\abs{x-y} : y \in \Omega}$ the distance from $x$ to $\Omega$, and $b_\Omega(x) = d_\Omega(x) - d_{\Omega^c}(x)$ is called the oriented distance function.
	By assumption on the regularity of the boundary, there exists a constant $\rho_0 > 0$ such that, on the collar domain $\cbr{x: d_{\partial\Omega}(x) < \rho_0}$, $b_\Omega$ is twice continuously differentiable with bounded derivatives of first and second order.

	The space $\Gamma$ will be defined as the set of all absolutely continuous curves $\gamma:[0,T] \to \Omega$ such that $\int_0^T \abs{\dot{\gamma}(t)}^2 \dif t < \infty$; this is the set of all admissible curves.
	We endow $\Gamma$ with the topology of uniform convergence, i.e.~ it is a subspace of the Banach space $C\del{[0,T];\Omega}$.
	For a given initial condition $x \in \Omega$ we set $\Gamma_x = \{\gamma \in \Gamma : \gamma(0) = x\}$.
	For a given time $t \in [0,T]$ we define the evaluation map $e_t:\Gamma \to \Omega$ by $e_t(\gamma) = \gamma(t)$.
	
	If $\eta \in \sr{P}(\Gamma)$, then for $t \in [0,T]$ we define $m^\eta(t) \in \sr{P}(\Omega)$ by $m^\eta(t) = e_t\sharp \eta$.
	Fix an initial measure $m_0 \in \sr{P}(\Omega)$.
	Define $\sr{P}_{m_0}(\Gamma)$ to be the set of all $\eta \in \sr{P}(\Gamma)$ such that $e_0 \sharp \eta = m_0$.
	For any $\gamma \in \Gamma$ and $\eta \in \sr{P}_{m_0}(\Gamma)$, define $J(\eta,\gamma)$ as in \eqref{J}, which we restate here:
	\begin{equation}
		J(\eta,\gamma) = \int_0^T \del{L\del{\gamma(t),\dot{\gamma}(t)} + f\del{m^\eta(t),\gamma(t)}}\dif t + g\del{m^\eta(T),\gamma(T)}.
	\end{equation}	
	We will need some assumptions on the Lagrangian $L$ and the couplings $f$ and $g$.
	\begin{assumption} \label{as:L}
		$L:\Omega \times \bb{R}^d \to \bb{R}$ is continuously differentiable, and it is convex with respect to the second variable.
		There exists a constant $C > 0$ such that
		\begin{equation} \label{eq:quad growth}
			\frac{1}{C}\abs{v}^2 -C \leq L(x,v) \leq C\abs{v}^2 + C \quad \forall v \in \bb{R}^d, x \in \Omega.
		\end{equation}
	\end{assumption}
	\begin{assumption} \label{as:fg}
		$f$ and $g$ are continuous real-valued functions on $\sr{P}_2(\Omega) \times \Omega$.
		Moreover, there exist functions $F$ and $G$ on $\sr{P}(\Omega)$ such that $f$ and $g$ are linear derivatives of $F$ and $G$, respectively, and $F$ and $G$ are bounded below.
	\end{assumption}
	
	We will define the set of best responses $\Gamma_x(\eta)$ to be the set of all $\gamma^* \in \Gamma_x$ such that
	\begin{equation}
		J(\gamma^*,\eta) \leq J(\eta,\gamma) \quad \forall \gamma \in \Gamma_x.
	\end{equation}
	We say that $\eta \in \sr{P}_{m_0}(\Gamma)$ is a \emph{Nash equilibrium} with respect to $J$ provided that
	\begin{equation} \label{Nash equilibrium}
		\operatorname{supp} \eta \subset \bigcup_{x \in \Omega} \Gamma_x(\eta).
	\end{equation}
	In other words, for $\eta$-a.e.~$\gamma^*$, $\gamma^*$ is a best response to $\eta$ over all curves starting at $x = \gamma^*(0)$.
	
	The potential is given by Equation \eqref{potential}, which we restate here:
	\begin{equation}
		\s{J}(\eta) = \int_\Gamma \int_0^T L\del{\gamma(t),\dot{\gamma}(t)} \dif t \dif \eta(\gamma) + \int_0^T F\del{m^\eta(t)}\dif t + G\del{m^\eta(T)}.
	\end{equation}
	Our goal is to show that minimizers of $\s{J}$ must be Nash equilibria, i.e.~they must satisfy \eqref{Nash equilibrium}.

	It follows from Assumptions \ref{as:L} and \ref{as:fg} that any minimizer $\eta$ of $\s{J}$ in $\sr{P}_{m_0}(\Gamma)$ must satisfy an estimate of the form
	\begin{equation} \label{eta moment}
		\int_{\Gamma} \int_0^T \abs{\dot{\gamma}(t)}^2 \dif t \dif \eta(\gamma) < \infty.
	\end{equation}
	We will denote by $\sr{P}^*_{m_0}(\Gamma)$ the set of all $\eta \in \sr{P}_{m_0}(\Gamma)$ that satisfy \eqref{eta moment}.
	\begin{lemma}
		\label{lem:cty of m(t)}
		Let $\eta \in \sr{P}^*_{m_0}(\Gamma)$ for some $m_0 \in \sr{P}_2(\Omega)$.
		Then $[0,T] \ni t \mapsto m^\eta(t) \in \sr{P}_2(\Omega)$ is continuous.
	\end{lemma}

	\begin{proof}
		The proof is similar to  \cite[Lemma 3.2]{cannarsa2018existence}, but in general we could be on an unbounded domain $\bb{R}^d$, so we must take this into account.
		
		Let $t_n \to t$ in $[0,T]$.
		It is enough to show that $m^\eta(t_n) \to m^\eta(t)$ narrowly.
		That is, let $\varphi(x)$ be a continuous function satisfying quadratic growth, i.e.
		\begin{equation}
			\abs{\varphi(x)} \leq C\del{1+\abs{x}^2} \forall x \in \Omega,
		\end{equation}
		for some constant $C$.
		We must show that $\int_{\Omega} \varphi(x)m^\eta(t_n,\dif x) \to \int_{\Omega} \varphi(x)m^\eta(t,\dif x)$, which is equivalent to showing that $\int_{\Gamma} \varphi\del{\gamma(t_n)}\dif \eta(\gamma) \to \int_{\Gamma} \varphi\del{\gamma(t)}\dif \eta(\gamma)$.
		Since $\varphi$ is continuous, for any continuous path $\gamma \in \Gamma$ we have $\varphi\del{\gamma(t_n)} \to \varphi\del{\gamma(t)}$.
		Let $B_R$ be the set of all $\gamma \in \Gamma$ such that $\max_t \abs{\gamma(t)} < R$.
		On $B_R$, we have that $\abs{\varphi\del{\gamma(t_n)}} \leq \max_{\abs{x} \leq R} \abs{\varphi(x)}$.
		By the dominated convergence theorem, it follows that $\int_{B_R} \varphi\del{\gamma(t_n)}\dif \eta(\gamma) \to \int_{B_R} \varphi\del{\gamma(t)}\dif \eta(\gamma)$.
		To finish the proof, it is enough to show that
		\begin{equation}
			\lim_{R \to \infty} \sup_n \abs{\int_{\Gamma \setminus B_R} \del{\varphi\del{\gamma(t_n)} - \varphi\del{\gamma(t)}}\dif \eta(\gamma)} = 0.
		\end{equation}
		We start by observing that
		\begin{equation}
			\abs{\int_{\Gamma \setminus B_R} \del{\varphi\del{\gamma(t_n)} - \varphi\del{\gamma(t)}}\dif \eta(\gamma)}
			\leq C\int_{\Gamma \setminus B_R} \del{2 + \abs{\gamma(t_n)}^2 + \abs{\gamma(t)}^2}\dif \eta(\gamma).
		\end{equation}
		Using the fundamental theorem of calculus and the Cauchy-Schwarz inequality,
		\begin{equation}
			\abs{\gamma(t)}^2 = \abs{\gamma(0) + \int_0^t \dot\gamma(s)\dif s}^2 \leq 2\abs{\gamma(0)}^2 + 2T\int_0^T \abs{\dot\gamma(s)}^2 \dif s,
		\end{equation}
		and so
		\begin{equation}
			\abs{\int_{\Gamma \setminus B_R} \del{\varphi\del{\gamma(t_n)} - \varphi\del{\gamma(t)}}\dif \eta(\gamma)}
			\leq C\int_{\Gamma \setminus B_R} \del{2 + 4\abs{\gamma(0)}^2 + 4T\int_0^T \abs{\dot\gamma(s)}^2 \dif s}\dif \eta(\gamma).
		\end{equation}
		Observe that $\int_{\Gamma} \abs{\gamma(0)}^2 \dif \eta(\gamma) = \int_{\Omega} \abs{x}^2 \dif m_0(x) < \infty$.
		Since we also have \eqref{eta moment}, it follows that $\gamma \mapsto 2 + 4\abs{\gamma(0)}^2 + 4T\int_0^T \abs{\dot\gamma(s)}^2 \dif s$ is an integrable function with respect to $\eta$.
		Letting $R \to \infty$ and using the dominated convergence theorem, the claim follows.
	\end{proof}

	In light of Assumptions \ref{as:L} and \ref{as:fg} and Lemma \ref{lem:cty of m(t)}, $J(\eta,\gamma)$ is well-defined and finite for any $\gamma \in \Gamma$ and $\eta \in \sr{P}_{m_0}^*(\Gamma)$.

	\subsection{Minimizers are equilibria in the Lagrangian formulation} \label{sec:minimizers are equilibria}
	\begin{theorem} \label{thm:minimizers are equilibria}
		Suppose $\eta$ minimizes $\s{J}$.
		Then it is a Nash equilibrium with respect to $J$.
	\end{theorem}
	The proof of Theorem \ref{thm:minimizers are equilibria} has the same spirit as that which is found in \cite{sandholm2001potential}.
	In this reference, the objective functional depends on distributions over a finite control set, and so the problem is finite dimensional.
	Here the problem is infinite dimensional, because controls are chosen from an infinite (indeed infinite dimensional) set of curves.

	\begin{proof}
		We will first prove the theorem in the case where $\Omega$ is translation invariant, e.g.~when $\Omega$ is $\bb{R}^d$ or $\bb{T}^d$.
		
		The economic principle at work is that, because $\s{J}$ is a primitive of $J$, whenever two players currently in an equilibrium exchange strategies, this causes $\s{J}$ to increase.
		To make this precise, let $\eta$ minimize $\s{J}$ and take $\gamma^* \in \operatorname{supp} \eta$.
		Let $\gamma' \in \Gamma_{x}$ where $x =\gamma^*(0)$.
		We want to show that $J(\gamma^*,\eta) \leq J(\gamma',\eta)$.
		For $\varepsilon > 0$, let $B_\varepsilon = B_\varepsilon(\gamma^*)$ be the set of all $\gamma \in \Gamma$ such that
		\begin{equation}
			\sup_{t \in [0,T]} \abs{\gamma(t) - \gamma^*(t)} < \varepsilon,
		\end{equation}
		i.e.~the ball of radius $\varepsilon$ around $\gamma^*$ in the uniform norm topology.
		Note that $\eta(B_\varepsilon) > 0$.
		Define $S_\varepsilon:\Gamma \to \Gamma$ by
		\begin{equation}
			S_\varepsilon(\gamma) = \begin{cases}
				\gamma(0) - x + \gamma', &\text{if}~\gamma \in B_\varepsilon,\\
				\gamma, &\text{if}~\gamma \notin B_\varepsilon.
			\end{cases}
		\end{equation}		
		Now define
		\begin{equation}
			\tilde{\eta}_\varepsilon = S_\varepsilon \sharp \eta.
		\end{equation}
		Notice that $\tilde{\eta}_\varepsilon \in \sr{P}(\Gamma)$, and to check that $e_0 \sharp \tilde{\eta}_\varepsilon = e_0 \sharp \eta = m_0$ it suffices to observe that $S(\gamma)(0) = \gamma(0)$ for every $\gamma \in \Gamma$.
		A straightforward calculation shows that
		\begin{multline} \label{eq:sJ - sJ}
			\s{J}\del{\tilde{\eta}_\varepsilon} - \s{J}(\eta)
			= \int_{B_\varepsilon} \int_0^T \del{L\del{\gamma(0) - x + \gamma'(t),\dot{\gamma}'(t)} - L\del{\gamma(t),\dot{\gamma}(t)}}\dif t \dif \eta(\gamma)\\
			+ \int_0^T \int_0^1 \int_{B_\varepsilon} \del{f\del{m_\lambda^\eta(t),\gamma(0) - x + \gamma'(t)} - f\del{m_\lambda^\eta(t),\gamma(t)}}\dif \eta(\gamma)\dif \lambda \dif t\\
			+ \int_0^1 \int_{B_\varepsilon} \del{g\del{m_\lambda^\eta(T),\gamma(0) - x + \gamma'(T)} - g\del{m_\lambda^\eta(T),\gamma(T)}}\dif \eta(\gamma)\dif \lambda,\\
			m_\lambda^\eta(t) := m^\eta(t) + \lambda\del{m^{\tilde{\eta}_\varepsilon}(t) - m^\eta(t)},
		\end{multline} 
		where we have used the fact that $f$ and $g$ are continuous linear derivatives of $F$ and $G$, respectively.
		It is not hard to see that $m^{\tilde{\eta}_\varepsilon}(t) \to m^\eta(t)$ in $\sr{P}_2(\Omega)$.
		Indeed, for any continuous function $\phi$ with quadratic growth, we have
		\begin{multline}
			\int_{\Omega} \phi(x)m^{\tilde{\eta}_\varepsilon}(t,\dif x)
			= \int_{\Gamma \setminus B_\varepsilon} \phi\del{\gamma(t)}\dif \eta(\gamma) + \int_{B_\varepsilon} \phi\del{S_\varepsilon(\gamma)(t)}\dif \eta(\gamma)\\
			\to \int_{\Gamma} \phi\del{\gamma(t)}\dif \eta(\gamma)
			= \int_{\Omega} \phi(x)m^{\eta}(t,\dif x),
		\end{multline}
		since $S_\varepsilon(\gamma)$ is uniformly close to $\gamma'$ in the set $B_\varepsilon$.
		Moreover, we have a uniform bound on $m^{\tilde{\eta}_\varepsilon}(t)$ in $\sr{P}_2(\Omega)$, i.e.~there is a bounded set containing $m^{\tilde{\eta}_\varepsilon}(t)$ for all $\varepsilon$ small and for all $t \in [0,T]$, since
		\begin{equation}
			\int_{\Omega} \abs{x}^2 m^{\tilde{\eta}_\varepsilon}(t,\dif x)
			= \int_{\Gamma \setminus B_\varepsilon} \abs{\gamma(t)}^2\dif \eta(\gamma) + \int_{B_\varepsilon} \abs{S_\varepsilon(\gamma)(t)}^2\dif \eta(\gamma)
			\leq \int_{\Gamma \setminus B_\varepsilon} \abs{\gamma(t)}^2\dif \eta(\gamma)
			+ 2\eta(B_\varepsilon)\del{\varepsilon^2 + \abs{\gamma'(t)}^2}.
		\end{equation}
		
		Now the left-hand side of \eqref{eq:sJ - sJ} is non-negative, because $\eta$ is a minimizer.
		We will divide both sides by $\eta\del{B_\varepsilon}$ and let $\varepsilon \to 0$.
		From the continuity of $f$ and $g$, we deduce that
		\begin{multline} \label{epto01}
			\frac{1}{\eta\del{B_\varepsilon}}\int_0^T \int_0^1 \int_{B_\varepsilon} \del{f\del{m_\lambda^\eta(t),\gamma(0) - x + \gamma'(t)} - f\del{m_\lambda^\eta(t),\gamma(t)}}\dif \eta(\gamma)\dif \lambda \dif t\\
			+ \frac{1}{\eta\del{B_\varepsilon}}\int_0^1 \int_{B_\varepsilon} \del{g\del{m_\lambda^\eta(T),\gamma(0) - x + \gamma'(T)} - g\del{m_\lambda^\eta(T),\gamma(T)}}\dif \eta(\gamma)\dif \lambda\\
			\to \int_0^T \del{f\del{m^\eta(t),\gamma'(t)} - f\del{m^\eta(t),\gamma^*(t)}} \dif t
			+ \del{g\del{m^\eta(T),\gamma'(T)} - g\del{m^\eta(T),\gamma^*(T)}}.
		\end{multline}
		Similarly, by the continuity of $L$ in the first variable, we get
		\begin{equation} \label{epto02}
			\frac{1}{\eta\del{B_\varepsilon}}\int_{B_\varepsilon} \int_0^T L\del{\gamma(0) - x + \gamma'(t),\dot{\gamma}'(t)}\dif t \dif \eta(\gamma)
			\to \int_0^T L\del{\gamma'(t),\dot{\gamma}'(t)}\dif t.
		\end{equation}
		For the remaining term, we want to show that 
		\begin{equation}
			\label{epto03}
			\int_0^T L\del{\gamma^*(t),\dot{\gamma^*}(t)}\dif t
			\leq \liminf_{\varepsilon \to 0}
			\frac{1}{\eta\del{B_\varepsilon}}\int_{B_\varepsilon} \int_0^T L\del{\gamma(t),\dot{\gamma}(t)}\dif t \dif \eta(\gamma).
		\end{equation}
		Define
		\begin{equation}
			\Gamma^K = \{\gamma \in \Gamma : \int_0^T L\del{\gamma(t),\dot{\gamma}(t)}\dif t \leq K\}.
		\end{equation}
		By Assumption \ref{as:L}, there is a constant $C_K$ such that $\int_0^T \abs{\dot{\gamma}(t)}^2 \dif t \leq C_K$ for all $\gamma \in \Gamma^K$.
		Thanks to the convexity of $L(x,v)$ with respect to $v$, it follows from standard arguments (cf.~\cite[Theorem 1 from Section 8.2]{evans10}) that $\gamma \mapsto \int_0^T L\del{\gamma(t),\dot{\gamma}(t)}\dif t$ is lower semi-continuous on $\Gamma^K$.
		We deduce that, for any $K < \int_0^T L\del{\gamma^*(t),\dot{\gamma}^*(t)}\dif t$, there is an $\varepsilon > 0$ small enough such that $\Gamma^K \cap B_\varepsilon = \emptyset$, and therefore
		\begin{equation}
			\liminf_{\varepsilon \to 0}
			\frac{1}{\eta\del{B_\varepsilon}}\int_{B_\varepsilon} \int_0^T L\del{\gamma(t),\dot{\gamma}(t)}\dif t \dif \eta(\gamma) \geq K.
		\end{equation}
		Letting $K \uparrow \int_0^T L\del{\gamma^*(t),\dot{\gamma}^*(t)}\dif t$ we deduce \eqref{epto03}.
		Combining \eqref{epto01}, \eqref{epto02}, and \eqref{epto03}, we see from \eqref{eq:sJ - sJ} that
		\begin{multline}
			\int_0^T L\del{\gamma^*(t),\dot{\gamma^*}(t)}\dif t
			+ \int_0^T  f\del{m^\eta(t),\gamma^*(t)}\dif t
			+ g\del{m^\eta(T),\gamma^*(T)}\\
			\leq \int_0^T L\del{\gamma'(t),\dot{\gamma}'(t)}\dif t
			+
			\int_0^T f\del{m^\eta(t),\gamma'(t)} \dif t
			+ g\del{m^\eta(T),\gamma'(T)},
		\end{multline}
		i.e.~$J(\gamma^*,\eta) \leq J(\gamma',\eta)$, as desired.
		
		Now we deal with the case where there are state constraints, i.e.~when $\Omega$ is not translation invariant.
		Here we use some ideas from \cite{cannarsa2018existence}.
		In this case the curve $S_\varepsilon(\gamma)$ has to be constructed in such a way as to not leave $\Omega$.
		Namely, if $\gamma \in B_\varepsilon$, we take the projection of $\gamma(0) - x + \gamma'$ onto $\Omega$.
		We recall that $b_\Omega$ is $C^2_b$ smooth in a neighborhood of radius $\rho_0 > 0$ around $\partial\Omega$; we may assume that $\varepsilon < \rho_0$.
		Thus
		\begin{equation}
			S_\varepsilon(\gamma)(t) = \begin{cases}
				\gamma(0) - x + \gamma'(t) - d_\Omega\del{\gamma(0) - x + \gamma'(t)}Db_\Omega\del{\gamma(0) - x + \gamma'(t)}, &\text{if}~\gamma \in B_\varepsilon,\\
				\gamma(t), &\text{if}~\gamma \notin B_\varepsilon
			\end{cases}
		\end{equation}
		where the first formula reduces to $\gamma(0) - x + \gamma'(t)$ if it is already in $\Omega$.
		Note that we still have $\abs{S_\varepsilon(\gamma)(t) - \gamma'(t)} \leq \varepsilon$ for every $\gamma \in B_\varepsilon$.
		Therefore, every step in the prove above goes through except one: we need to prove that
		\begin{equation} \label{epto02'}
			\frac{1}{\eta\del{B_\varepsilon}}\int_{B_\varepsilon} \int_0^T L\del{S_\varepsilon(\gamma)(t),\od{}{t}S_\varepsilon(\gamma)(t)}\dif t \dif \eta(\gamma)
			\to \int_0^T L\del{\gamma'(t),\dot{\gamma}'(t)}\dif t
		\end{equation}
		as $\varepsilon \to 0$.
		
		Let $\gamma \in B_\varepsilon$.
		By \cite[Lemma 3.1]{cannarsa2018existence} we have that $S_\varepsilon(\gamma)$ is indeed absolutely continuous with
		\begin{multline}
			\od{}{t}S_\varepsilon(\gamma)(t) = \dot{\gamma}'(t) - \ip{Db_\Omega\del{\gamma(0) - x + \gamma'(t)}}{\dot{\gamma}'(t)}Db_\Omega\del{\gamma(0) - x + \gamma'(t)}{\bf 1}_{\Omega^c}\del{\gamma(0) - x + \gamma'(t)}\\
			- d_\Omega\del{\gamma(0) - x + \gamma'(t)}D^2 b_\Omega\del{\gamma(0) - x + \gamma'(t)}\dot{\gamma}'(t), \quad \text{a.e. } t \in [0,T],
		\end{multline}
		and as in \cite[Proposition 3.1]{cannarsa2018existence} there exists a constant $C > 0$ (independent of both $\varepsilon$ and $t$) such that
		\begin{equation} \label{leq Cgammadot}
			\abs{\od{}{t}S_\varepsilon(\gamma)(t)} \leq C\abs{\dot{\gamma}'(t)} \quad \forall t \in [0,T].
		\end{equation}
		For constants $\delta > 0$ (small) and $M > 0$ (large), define
		\begin{equation}
			\begin{split}
				I_0 &:= \cbr{t \in [0,T] : \gamma'(t) \in \partial\Omega},\\
				I_\delta' &:= \cbr{t \in [0,T] : 0 < d_\Omega\del{\gamma'(t)} < \delta},\\
				I_\delta &:= \cbr{t \in [0,T] : d_\Omega\del{\gamma'(t)} \geq \delta},\\
				J_M &:= \cbr{t \in [0,T]: \abs{\dot{\gamma}'(t)} \leq M}.
			\end{split}
		\end{equation}
		To prove \eqref{epto02'}, we will analyze $L\del{S_\varepsilon(\gamma)(t),\od{}{t}S_\varepsilon(\gamma)(t)}$ on each of these sets.
		If $\varepsilon < \delta$, then it follows that $S_\varepsilon(\gamma)(t) = \gamma(0) - x + \gamma'(t)$ for every $t \in I_\delta$, for every $\gamma \in B_\varepsilon$.
		We deduce that
		\begin{equation} \label{epto02'a}
			\frac{1}{\eta\del{B_\varepsilon}}\int_{B_\varepsilon} \int_{I_\delta} L\del{S_\varepsilon(\gamma)(t),\od{}{t}S_\varepsilon(\gamma)(t)}\dif t \dif \eta(\gamma)
			\to \int_{I_\delta} L\del{\gamma'(t),\dot{\gamma}'(t)}\dif t.
		\end{equation}
		Next, by \cite[Lemma 3.1]{cannarsa2018existence} we have that for a.e.~$t \in I_0$, $\ip{Db_\Omega\del{\gamma(0) - x + \gamma'(t)}}{\dot{\gamma}'(t)} = 0$, and therefore we can write
		\begin{multline}
			\od{}{t}S_\varepsilon(\gamma)(t) = \dot{\gamma}'(t)\\ - \ip{Db_\Omega\del{\gamma(0) - x + \gamma'(t)} - Db_\Omega\del{\gamma'(t)}}{\dot{\gamma}'(t)}Db_\Omega\del{\gamma(0) - x + \gamma'(t)}{\bf 1}_{\Omega^c}\del{\gamma(0) - x + \gamma'(t)}\\
			- d_\Omega\del{\gamma(0) - x + \gamma'(t)}D^2 b_\Omega\del{\gamma(0) - x + \gamma'(t)}\dot{\gamma}'(t), \quad \text{a.e. } t \in I_0,
		\end{multline}
		for any $\gamma \in B_\varepsilon$.
		Let $C_\Omega$ be a bound on $D^2 b_\Omega$ in $\partial \Omega + B_{\rho_0}$.
		Since $\abs{\gamma(0) - x} \leq \varepsilon$, we deduce that
		\begin{equation}
			\abs{\od{}{t}S_\varepsilon(\gamma)(t) - \dot{\gamma}'(t)}
			\leq 2C_\Omega \varepsilon\abs{\dot{\gamma}'(t)}, \quad \text{a.e. } t \in I_0.
		\end{equation}
		Using the fact that $L$ is locally Lipschitz (in fact continuously differentiable), there exists a constant $C = C(\Omega,M)$ such that for every $\varepsilon > 0$ small enough,
		\begin{equation}
			\abs{L\del{S_\varepsilon(\gamma)(t),\od{}{t}S_\varepsilon(\gamma)(t)} - L\del{\gamma'(t),\dot{\gamma}'(t)}}
			\leq C\varepsilon \del{1 + \abs{\dot{\gamma}'(t)}}
			\quad \forall t \in I_0 \cap J_M, \ \forall \gamma \in B_\varepsilon.
		\end{equation}
		We deduce that
		\begin{equation} \label{epto02'b}
			\frac{1}{\eta\del{B_\varepsilon}}\int_{B_\varepsilon} \int_{I_0 \cap J_M} L\del{S_\varepsilon(\gamma)(t),\od{}{t}S_\varepsilon(\gamma)(t)}\dif t \dif \eta(\gamma)
			\to \int_{I_0 \cap J_M} L\del{\gamma'(t),\dot{\gamma}'(t)}\dif t.
		\end{equation}
		Finally, combining the estimate in Assumption \ref{as:L} with \eqref{leq Cgammadot}, we get
		\begin{equation}
			\label{epto02'c}
			\frac{1}{\eta\del{B_\varepsilon}}\int_{B_\varepsilon} \abs{\int_{I_\delta' \cup J_M^c} L\del{S_\varepsilon(\gamma)(t),\od{}{t}S_\varepsilon(\gamma)(t)}\dif t \dif \eta(\gamma)}
			\leq C \int_{I_\delta' \cup J_M^c} \del{1 + \abs{\dot{\gamma}'(t)}^2}\dif t,
		\end{equation}
		for some constant $C$.
		Combining \eqref{epto02'a}, \eqref{epto02'b}, and \eqref{epto02'c}, then letting $\delta \to 0$ and $M \to \infty$, we deduce \eqref{epto02'}.
		This completes the proof.
	\end{proof}

	\subsection{Equivalence of Lagrangian and original formulations}
	
	At least in certain cases, the minimization of \eqref{sJ} over measures $\sr{P}_{m_0}(\Gamma)$ (i.e.~the Lagrangian formulation) can be proved to be equivalent to the optimal control of the continuity equation, namely, the minimization of \eqref{potential} subject to the constraint \eqref{continuity eq} and initial condition $m(0) = m_0$ (i.e.~the Eulerian formulation).
	This is proved in \cite{benamou2017variational} for quadratic Lagrangians $L(x,v) = \frac{1}{2}\abs{v}^2$ on compact domains $\Omega$.
	The key is \cite[Proposition 3.1]{benamou2017variational}, which proves a one-to-one correspondence between solutions to the continuity equation and admissible probability measures $\eta$ on the space of curves $\Gamma$.
	A similar result is found in \cite[Theorem 8.2.1]{ambrosio2008gradient} and seems to allow us to prove the that equivalence between Lagrangian and Eulerian formulations holds also when $\Omega = \bb{R}^d$, at least when $L(x,v) = \frac{1}{2}\abs{v}^2$.
	In this brief article, we will not pursue this issue further, since in the example studied in the following section, the Lagrangian formulation is sufficient.

	\subsection{When can a mean field game be a potential game?} \label{sec:whenispotential}
		
	In this subsection, we will consider whether or not it is possible to find a potential for a given mean field game.
	A result in this direction has been achieved for $N$-player dynamic games in \cite{guo2023towards}.
	All of our results will remain at the formal level; that is, we will not state precise assumptions on the data but assume they are as smooth as necessary for our conclusions to hold.
	Let us consider a more general cost than \eqref{J} given by
	\begin{equation} \label{J1}
		J(\eta,\gamma) = \int_0^T L\del{\mu^\eta(t),\gamma(t),\dot{\gamma}(t)}\dif t + g\del{m^\eta(T),\gamma(T)}
	\end{equation}
	where $\mu^\eta(t)$ is defined as the distribution of states and velocities at time $t$ for curves distributed according to $\eta$.
	To be precise, let us restrict to measures $\eta$ on the space $\Gamma' = C^1\del{[0,T];\Omega}$ of continuously differentiable functions taking values in $\Omega$.
	Then let $\dot{e}_t(\gamma) := \dot{\gamma}(t)$, and set $\mu^\eta(t) = (e_t,\dot{e}_t) \sharp \eta$.
	We define Nash equilibrium in the same way as before.
	At this level of generality, the game is often referred to as a \emph{mean field game of controls}, reflecting the fact that the distribution of controls also plays a role \cite{cardaliaguet2017mfgcontrols}; the terminology ``extended mean field games'' has also been used \cite{gomes2014extended,gomes2016extended}.
	Potential mean field games of controls have been considered in \cite{graber2018variational,bonnans2019schauder,graber2021weak}.
	
	If $J$ defined in \eqref{J1} has a potential in the same way \eqref{J} has potential given by \eqref{sJ}, it must have the form
	\begin{equation}
		\s{J}(\eta) = \int_0^T \sr{L}\del{\mu^\eta(t)}\dif t + G\del{m^\eta(T)}.
	\end{equation}
	Formally,
	\begin{equation}
		\vd{\s{J}}{\eta}(\eta,\gamma) = \int_0^T \vd{\sr{L}}{\mu}\del{\mu^\eta(t),\gamma(t),\dot{\gamma}(t)}\dif t + \vd{G}{m}\del{m^\eta(T),\gamma(T)},
	\end{equation}
	where the notation $\vd{\s{J}}{\eta}$ refers to ``the'' linear derivative of $\s{J}$ with respect to the measure variable (which is unique only up to a constant).
	If $\s{J}$ is a potential, it follows that $L(\mu,x,v)$ must be a linear derivative of $\sr{L}(\mu)$, so we may write
	\begin{equation}
		\vd{\sr{L}}{\mu}\del{\mu,x,v} = L(\mu,x,v).
	\end{equation}
	(In fact, this equality need only hold up to addition of a constant depending only on $\mu$.)
	Taking another derivative, we have
	\begin{equation}
		\vd{}{\mu}\vd{\sr{L}}{\mu}\del{\mu,x,v,\tilde{x},\tilde{v}} = \vd{L}{\mu}(\mu,x,v,\tilde{x},\tilde{v}).
	\end{equation}
	By the usual arguments, if the left-hand side is continuous in all variables, then it is also symmetric in the sense
	\begin{equation}
		\vd{}{\mu}\vd{\sr{L}}{\mu}\del{\mu,x,v,\tilde{x},\tilde{v}} = \vd{}{\mu}\vd{\sr{L}}{\mu}\del{\mu,\tilde{x},\tilde{v},x,v},
	\end{equation}
	which implies
	\begin{equation} \label{symmetry}
		\vd{L}{\mu}(\mu,x,v,\tilde{x},\tilde{v}) = \vd{L}{\mu}(\mu,\tilde{x},\tilde{v},x,v).
	\end{equation}
	This symmetry consideration can be used to provide many examples of Lagrangians $L$ for which there cannot be a potential.
	For example, suppose that $L$ depends only on the distribution of states and not controls, i.e.~$L(\mu,x,v) = \ell(m,x,v)$ where $m = p^1 \sharp \mu$ and $p^1(x,v) = x$.
	We have
	\begin{equation}
		\vd{L}{\mu}(\mu,x,v,\tilde{x},\tilde{v}) = \vd{\ell}{m}(m,x,v,\tilde{x}).
	\end{equation}
	By the symmetry criterion \eqref{symmetry}, it follows that $\vd{\ell}{m}(m,x,v,\tilde{x})$ does not depend on $v$.
	We now deduce that $\ell(m,x,v)$ can be written in the separated form
	\begin{equation}
		\ell(m,x,v) = l(x,v) + f(m,x)
	\end{equation}
	for some functions $l$ and $f$.
	Hence if the particular case we are considering is \emph{not} a mean field game of controls, then it is a potential game if \emph{and only if} the Lagrangian can be written in separated form.
	
	\section{Remarks on the selection problem} \label{sec:selection}
	
	The problem of equilibrium selection is well-known in game theory.
	In the Foreword to Harsanyi and Selten's 1988 text \cite{harsanyi1988general}, Aumann discusses the justification of the equilibrium concept as follows:
	\begin{quotation}
		``One of the oldest rationales for this concept, advanced already by von Neumann and Morgenstern (1944), is that any normative theory that advises players how to play games must pick an equilibrium in each game.
		A theory recommending anything other than an equilibrium would be self-defeating, in the sense that a player who believes that the others are following the theory will sometimes be motivated to deviate from it.''
	\end{quotation}
	However, he then goes on to remark how this justification relates to selection:
	\begin{quotation}
		``In general, a given game may have several equilibria.
		Yet uniqueness is crucial to the foregoing argument.
		Nash equilibrium makes sense only if each player knows which strategies the others are playing; if the equilibrium recommended by the theory is not unique, the players will not have this knowledge.
		Thus it is essential that for each game, the theory selects one unique equilibrium from the set of all Nash equilibrium.
	\end{quotation}
	Selten introduced what is known as ``trembling hand perfect'' equilibrium \cite{selten1975}.
	The idea is to keep only those equilibria that are robust with respect to perturbation of the game by noise.
	In the case where the set of available controls is finite, one can do this simply by assigning a uniformly small probability $\varepsilon$ that any given player might switch to a different strategy, and then see whether, for $\varepsilon$ small enough, the given strategy distribution is still optimal for each player.
	A series of further refinements in this vein have been proposed in \cite{myerson1978refinements,kalai1984persistent,hillas1990definition,harsanyi1995new}; see \cite{van2012refinements} for an overview.
	
	The basic idea of selecting solutions that are robust with respect to perturbations by noise is well-known in the study of PDE.
	Indeed, the concept of an entropy solution to a nonlinear transport equation can be seen as the solution selected by the vanishing viscosity limit (e.g.~\cite{kruvzkov1970first}), which can also be thought of as a vanishing perturbation by Brownian noise.
	In the discussion below, we will see that the same type of selection principle can be used for mean field games under certain conditions.
	Note that the strategy space is infinite, so a uniform perturbation is impossible; instead, a Gaussian perturbation can be seen as a natural choice.
	In the context of potential mean field games, the ``vanishing noise limit'' will be seen to correspond to strategies that minimize the potential.
	
	Throughout this section we will consider only a very particular case of the mean field game studied above.
	This will allow us to develop the main ideas without getting caught up in too many technical details, most of which are still poorly understood.
	The main ideas are as follows.
	Corresponding to the problem of minimizing the potential, there exists a Hamilton-Jacobi equation on the space of probability measures.
	Formally taking the derivative of this equation leads to the so-called \emph{master equation} in mean field games \cite{cardaliaguet2019master,delarue2019master,chassagneux2022probabilistic}, which corresponds to finding a Nash equilibrium.
	When the Nash equilibrium is unique, there is a unique classical solution to the master equation (see \cite{gangbo2022global} for a result that applies directly to our present context).
	When it is not unique, a selection principle must be introduced, which is deeply connected to defining a notion of weak solution to the master equation.
	We will see that the notion of entropy solution is an attractive candidate, but it is not without difficulties.

	
	\subsection{Preliminaries}
	
	Let $B$ be the ball of unit volume around the origin in $\bb{R}^d$, let $\sr{B}$ be the Borel $\sigma$-algebra on $B$, and let $\s{L}^d$ be $d$-dimensional Lebesgue measure; note that $(B,\sr{B},\s{L}^d)$ is a probability space.
	Define $\bb{H} := L^2(B,\sr{B},\s{L}^d)$, and observe that $\bb{H}$ is a separable Hilbert space with inner product given by $\ip{X}{Y} = \bb{E}\sbr{X \cdot Y}$ and norm $\enVert{X} = \ip{X}{X}^{1/2} = \bb{E}[\abs{X}^2]^{1/2}$.
	For $X \in \bb{H}$ we define $\sharp(X) = X\sharp \s{L}^d$, which is the law of the random variable $X$.
	By Brenier's polar factorization result \cite{brenier1991polar}, we have that $\sharp:\bb{H} \to \sr{P}_2(\bb{R}^d)$ is surjective \cite{gangbo2019differentiability}.
	Moreover, we have
	\begin{equation} \label{eq:rv Wasserstein}
		W_2(\mu,\nu)^2 = \min\cbr{\bb{E}\abs{X-Y}^2 : X,Y \in \bb{H}, \sharp(X) = \mu, \ \sharp(Y) = \nu}.
	\end{equation}
	Let $F:\sr{P}_2(\bb{R}^d) \to \bb{R}$.
	We define $\tilde{F}:\bb{H} \to \bb{R}$ by $\tilde{F}(X) = F\del{\sharp(X)}$, and we refer to $\tilde{F}$ as the ``lifted version'' of $F$, because it lifts the domain from $\sr{P}_2(\bb{R}^d)$, which is a curved space, to $\bb{H}$, which is a Hilbert space.
	We say that $F$ has a Wasserstein gradient $D_\mu F:\sr{P}_2(\bb{R}^d) \times \bb{R}^d \to \bb{R}$ provided $\tilde F$ is continuously differentiable with Fr\'echet gradient $D\tilde F:\bb{H} \to \bb{H}$, such that
	\begin{equation}
		D\tilde F(X) = D_\mu F(\sharp(X),X).
	\end{equation}
	The following result, which is \cite[Proposition 5.48]{carmona2017probabilistic} and \cite[Proposition 5.51]{carmona2017probabilistic}, relates the Wasserstein gradient to the linear derivative.
	\begin{proposition} \label{pr:was lin}
		Suppose $F$ has a linear derivative $f(m,x)$ that is differentiable with respect to $x$ such that both $f(m,x)$ and $D_x f(m,x)$ are continuous in both variables.
		Assume $D_x f(m,x)$ has at most linear growth in $x$, and in particular there exists for each $R > 0$ a constant $C(R)$ such that
		\begin{equation}
			\abs{D_x f(m,x)} \leq C(R)\del{1 + \abs{x}} \quad \forall x \in \bb{R}^d, \quad \forall m \in \sr{P}_2(\bb{R}^d) \ \text{s.t.} \ \int \abs{\xi}^2 \dif m(\xi) \leq R.
		\end{equation}
		Then $F$ has a continuous Wasserstein gradient given by
		\begin{equation} \label{eq:was lin}
			D_\mu F(m,x) = D_x f(m,x).
		\end{equation}
		Conversely, if $F$ has a continuous Wasserstein gradient, and if the Fr\'echet gradient $D\tilde F$ is Lipschitz continuous, then $F$ has a linear derivative $f(m,x)$ that is continuous in both variables, and the relation \eqref{eq:was lin} holds.
	\end{proposition}
	\begin{proof}
		See \cite[Propositions 5.48 and 5.51]{carmona2017probabilistic}, where the Wasserstein gradient is referred to as the L-derivative.
	\end{proof}
	For a further comparison of different notions of differentiability on the Wasserstein space, see \cite[Chapter 5]{carmona2017probabilistic} and \cite{gangbo2019differentiability}.
	
	We say that a measure $m$ is \emph{empirical} (or \emph{discrete}) if it is the weighted sum of Dirac masses.
	We will often deal with the particular case of an empirical measure with $n$ points equally weighted: for each vector ${\bf x} = (x_1,\ldots,x_n) \in (\bb{R}^{d})^n$, we write $m_{\bf x}^n = \frac{1}{n}\sum_{j=1}^n \delta_{x_j}$.
	We have the following ``isometry'' relation:
	\begin{equation}
		W_2(m_{\bf x}^n,m_{\bf y}^n)^2 = \min\cbr{\frac{1}{n}\sum_{j=1}^n \abs{x_j - y_{\sigma(j)}}^2 : \sigma \in S_n}
	\end{equation}
	where $S_n$ is the permutation group on $\{1,\ldots,n\}$.
	If we can relabel the indices of ${\bf y}$ so that they are in the optimal order with respect to ${\bf x}$, we can then write $W_2(m_{\bf x}^n,m_{\bf y}^n) = n^{-1/2}\abs{{\bf x} - {\bf y}}$.
	
	For a function $F:\sr{P}_2(\bb{R}^d) \to \bb{R}$, we will define $F_n: (\bb{R}^d)^n \to \bb{R}$ by $F_n({\bf x}) = F(m_{\bf x}^n)$.
	It is straightforward to show that
	\begin{equation}\label{eq:finite derivative}
		D_{x_j}F_n({\bf x}) = \frac{1}{n}D_\mu F(m_{\bf x}^n,x_j).
	\end{equation}
	See \cite[Proposition 5.35]{carmona2017probabilistic}.
	
	We say that $F$ is \emph{displacement convex} provided the lifted version $\tilde F(X)$ is convex; in this case we also have that $F_n({\bf x})$ is a convex function on $(\bb{R}^d)^n$.
	We say that a continuous function $f:\sr{P}_2(\bb{R}^d) \times \bb{R}^d \to \bb{R}$ with continuous derivative $D_x f$ is \emph{displacement monotone} provided that for every $X,Y \in \bb{H}$ we have
	\begin{equation}
		\bb{E}\sbr{\del{D_xf(\sharp(X),X) - D_xf(\sharp{Y},Y)}(X-Y)} \geq 0.
	\end{equation}
	If $D_x f = D_\mu F$, then $F$ is displacement convex if and only if $f$ is displacement monotone \cite{gangbo2022mean}.
	Essentially, this follows from the fact that $\tilde F$ is convex if and only if $D\tilde F$ is monotone; cf.~Proposition \ref{pr:was lin}.
	
	Another type of monotonicity, which is more common in mean field game theory, is called the Lasry-Lions (LL) condition \cite{lasry07}.
	We say that $f:\sr{P}_2(\bb{R}^d) \times \bb{R}^d \to \bb{R}$ is LL monotone provided that
	\begin{equation}
		\int_{\bb{R}^d} \del{f(m_1,x)-f(m_2,x)}\dif(m_1-m_2)(x) \geq 0 \quad \forall m_1,m_2 \in \sr{P}_2(\bb{R}^d),
	\end{equation}
	and we say it is \emph{strictly} LL monotone provided that equality implies $f(m_1,x) = f(m_2,x)$ for all $x$.
	See \cite{graber2023monotonicity} for a further discussion of monotonicity in mean field games.
	Note that if $f$ is a linear derivative of $F:\sr{P}_2(\bb{R}^d) \to \bb{R}$, then $f$ is LL monotone if and only if $F$ is convex in the usual sense, i.e.~$F(tm_1 + (1-t)m_2) \leq tF(m_1) + (1-t)F(m_2)$ for all $m_1,m_2 \in \sr{P}_2(\bb{R}^d)$ and $t \in [0,1]$.

	\subsection{A reduced problem}
	Throughout the remainder of this section, let us assume that the running cost $f \equiv 0$ and $L(x,v) = \frac{1}{2}\abs{v}^2$, that is,
	\begin{equation} \label{Jsimple}
		J(\eta,\gamma) = \frac{1}{2}\int_0^T \abs{\dot{\gamma}(t)}^2 \dif t + g\del{m^\eta(T),\gamma(T)}.
	\end{equation}
	The potential reduces to
	\begin{equation} 
		\s{J}(\eta) = \frac{1}{2}\int_\Gamma \int_0^T \abs{\dot{\gamma}(t)}^2 \dif t \dif \eta(\gamma) + G\del{m^\eta(T)}.
	\end{equation}
	We will assume for simplicity that $G$ is has a continuous Wasserstein gradient $D_\mu G(m,x)$ and that the gradient $D\tilde G(X)$ is Lipschitz continuous.
	It follows from Proposition \ref{pr:was lin} that $g(m,x)$ and $D_x g(m,x)$ are continuous and that $D_\mu G = D_x g$.
	We will also assume that the state space is $\Omega = \bb{R}^d$.
	
	To minimize $\s{J}$ defined in \eqref{Jsimple}, it is sufficient to travel in straight lines from the initial measure $m_0$, as the following lemma makes clear.
	For $m_T \in \sr{P}_2(\bb{R}^d)$ we define
	\begin{equation}
		\s{K}(m_T) = \frac{1}{2T}W_2\del{m_0,m_T}^2 + G\del{m_T}.
	\end{equation}
	\begin{lemma} \label{lem:straight lines}
		$\inf\cbr{\s{J}(\eta) : \eta \in \sr{P}_{m_0}(\Gamma)} = \inf\cbr{\s{K}(m_T) : m_T \in \sr{P}_2(\bb{R}^d)}$, i.e.~to minimize $\s{J}(\eta)$ it suffices to consider $\eta$ concentrated on straight lines.
	\end{lemma}
	
	\begin{proof}
		First, for $\pi \in \sr{P}_2(\bb{R}^d \times \bb{R}^d)$ we define
		\begin{equation}
			\s{K}'(\pi) = \frac{1}{2T}\int_{\bb{R}^d \times \bb{R}^d} \abs{x-y}^2 \dif \pi(x,y) + G\del{p^2 \sharp \pi}
		\end{equation}
		where $p^2$ is the projection $p^2(x,y) = y$.
		We claim 
		\begin{equation} \label{eq:J = K'}
			\inf\cbr{\s{J}(\eta) : \eta \in \sr{P}_{m_0}(\Gamma)} = \inf\cbr{\s{K}'(\pi) : \pi \in \sr{P}(\bb{R}^d \times \bb{R}^d) : p^1 \sharp \pi = m_0}
		\end{equation}
		where $p^1(x,y) = x$.
		For any $(x,y) \in \bb{R}^d \times \bb{R}^d$, we set
		\begin{equation}
			\gamma_{x,y}(t) = x + \frac{t}{T}(y-x).
		\end{equation}
		The map $S:\bb{R}^d \times \bb{R}^d \to \Gamma$ given by $S(x,y) = \gamma_{x,y}$ is continuous.
		For any $\pi \in \sr{P}_2(\bb{R}^d \times \bb{R}^d)$ we see that $\s{J}(S \sharp \pi) = \s{K}'(\pi)$, using the fact that $\dot{\gamma}(t) = 1/T$.
		Moreover, if $p^1 \sharp \pi = m_0$, then $S \sharp \pi \in \sr{P}_{m_0}(\Gamma)$ since $e_0 \circ S = p^1$.
		It follows immediately that \eqref{eq:J = K'} holds with $\leq$ in place of $=$.
		To prove the opposite inequality, let $\eta \in \sr{P}_{m_0}(\Gamma)$.
		By Jensen's inequality and the fundamental theorem of calculus,
		\begin{equation}
			\s{J}(\eta) \geq \frac{1}{2T}\int_{\Gamma} \abs{\gamma(T) - \gamma(0)}^2\dif \eta(\gamma) + G\del{m^\eta(T)}
			= \frac{1}{2T}\int_{\Gamma} \abs{y-x}^2\dif \del{(e_0,e_T)\sharp \eta}(x,y) + G\del{e_T \sharp \eta},
		\end{equation}
		i.e.~$\s{J}(\eta) \geq \s{K}'(\pi)$ where $\pi = (e_0,e_T) \sharp \eta$.
		This proves \eqref{eq:J = K'}.
		
		We now argue that
		\begin{equation} \label{eq:K' = K}
			\inf\cbr{\s{K}'(\pi) : \pi \in \sr{P}(\bb{R}^d \times \bb{R}^d) : p^1 \sharp \pi = m_0}
			= \inf\cbr{\s{K}(m_T) : m_T \in \sr{P}_2(\bb{R}^d)}.
		\end{equation}
		If $\pi \in \sr{P}(\bb{R}^d \times \bb{R}^d)$ with $p^1 \sharp \pi = m_0$, then letting $m_T = p^2 \sharp \pi$ yields $\int_{\bb{R}^d \times \bb{R}^d} \abs{x-y}^2 \dif \pi(x,y) \geq W_2(m_0,m_T)^2$ and thus $\s{K}'(\pi) \geq \s{K}(m_T)$.
		Conversely, if $m_T \in \sr{P}_2(\bb{R}^d)$, then we can find an optimal transport plan $\pi$ with $p^1 \sharp \pi = m_0$ and $p^2 \sharp \pi = m_T$ such that $\int_{\bb{R}^d \times \bb{R}^d} \abs{x-y}^2 \dif \pi(x,y) = W_2(m_0,m_T)^2$ and thus $\s{K}'(\pi) = \s{K}(m_T)$.
		Equation \eqref{eq:K' = K} follows, which is what we needed to show.
	\end{proof}
	
	Fix an arbitrary $X \in \bb{H}$ whose law is $m_0$, and define for any $Y \in \bb{H}$ the functional
	\begin{equation} \label{tildeK}
		\tilde{\s{K}}(Y) = \frac{1}{2T}\bb{E}\abs{X-Y}^2 + G\del{\sharp(Y)}.
	\end{equation}
	The proof of the following lemma is found, for example, in the lectures by Lions \cite{lions07}, see \cite[Remark 6.9]{cardaliaguet2010notes}.
	\begin{lemma} \label{lem:lift K}
		$\inf\cbr{\s{K}(m_T) : m_T \in \sr{P}_2(\bb{R}^d)} = \inf\cbr{\tilde{\s{K}}(Y) : Y \in \bb{H}}$.
	\end{lemma}
	\begin{proof}
		For the reader's convenience, we give some details of the proof.
		If $Y \in \bb{H}$, set $m_T = \sharp(Y)$; from \eqref{eq:rv Wasserstein} we get $\tilde{\s{K}}(Y) \geq \s{K}(m_T)$.
		Conversely, let $m_T \in \sr{P}_2(\bb{R}^d)$ be given.
		There exist $X',Y' \in \bb{H}$ such that $\bb{E}\abs{X'-Y'}^2 = W_2(m_0,m_T)$, $\sharp(X') = m_0$ and $\sharp(Y') = m_T$.		
		Since $X'$ and $X$ have the same law, there exists a bijection $S:B \to B$  such that $S$ and $S^{-1}$ are both measure preserving and $X' = X \circ S$.
		Now let $Y = Y' \circ S^{-1}$, so that $\sharp(Y) = \sharp(Y') = m_T$ and $\bb{E}\abs{X-Y}^2 = \bb{E}\abs{X' - Y'}^2 = W_2(m_0,m_T)^2$.
		It follows that $\s{K}(m_T) = \tilde{\s{K}}(Y)$.
		This completes the proof.
	\end{proof}

	Before continuing, we remark that it is not difficult to prove the existence of minimizers (and therefore equilibria) by assuming $G$ is also continuous in a slightly stronger sense, e.g.~$G:\sr{P}_p(\bb{R}^d) \to \bb{R}$ is continuous for some $p < 2$.
	In this case one can appeal to the arguments in \cite[Section 4]{hynd2015value}.
	More directly, suppose $m_{T,n}$ is a minimizing sequence for $\s{K}$.
	By the assumptions on $G$ we deduce that
	\begin{equation}
		G(m_{T,n}) \leq C\del{1 + \del{\int_{\bb{R}^d} \abs{x}^2 \dif m_{T,n}}^{1/2}}
	\end{equation}
	for some constant $C$.
	Using the inequality
	\begin{equation}
		\int_{\bb{R}^d} \abs{x}^2 \dif m_{T,n} \leq 2W_2(m_0,m_{T,n})^2 + 2\int_{\bb{R}^d} \abs{x}^2 \dif m_0,
	\end{equation}
	we deduce that there is a uniform bound on $\int_{\bb{R}^d} \abs{x}^2 \dif m_{T,n}$.
	Thus $\{m_{T,n}\}$ is tight and has uniformly integrable $p$th moments for $p < 2$.
	By Prokhorov's Theorem and \cite[Proposition 7.1.5]{ambrosio2008gradient} we can pass to a subsequence that converges in $\sr{P}_p(\bb{R}^d)$ to some $m_T \in \sr{P}_2(\bb{R}^d)$.
	Using the weak lower-semicontinuity of the $W_2$ metric \cite[Lemma 7.1.4]{ambrosio2008gradient}, we deduce that $m_T$ is a minimizer.
	Let us refer to \cite{gomes2015minimizers,gomes2016infinite} for related results on the existence of minimizers to infinite dimensional optimal control problems.
	
	\subsection{Necessary and sufficient conditions for uniqueness of minimizers and equilibria}

	We are now ready to state an important result relating uniqueness of minimizers to the potential to the uniqueness of a Nash equilibrium.
	The result we will prove is certainly not as general as possible.
	However, it communicates an important point: displacement monotonicity of the coupling is \emph{necessary} as well as \emph{sufficient} for uniqueness of the equilibrium.
	By contrast, Lasry-Lions monotonicity is a sufficient condition for uniqueness of the \emph{value function} for Nash equilibrium, i.e.~the function $u$ in System \eqref{mfg system}.
	See e.g.~\cite{cannarsa2019constrained} for a proof in the case of ``mild'' solutions.
	See also \cite{achdou2020introduction}. 
	On the other hand, by \cite[Proposition 3.5]{graber2023monotonicity}, LL monotonicity does not imply uniqueness of the equilibrium distribution, for the simple reason that mass can split in different directions.
	Intuitively, uniqueness of the cost to players does not imply uniqueness of the paths they take to reach that cost; this requires convexity of the cost $g(m,x)$ with respect to the state variable $x$, which holds if it is displacement monotone.
	
	\begin{theorem} \label{thm:uniqueness}
		In addition to the regularity assumed above, suppose that $g(m,x)$ is convex with respect to the $x$-variable.
		The following are equivalent:
		\begin{enumerate}
			\item $G$ is displacement convex.
			\item $g$ is displacement monotone.
			\item For every $T > 0$, $\s{J}$ has a unique minimizer.
			\item For every $T > 0$, there is a unique Nash equilibrium for the cost $J$ given in \eqref{Jsimple}.
		\end{enumerate}
	\end{theorem}

	Before getting to the proof of Theorem \ref{thm:uniqueness}, we make a few remarks.
	First, it is natural to conjecture that the theorem holds even without the added assumption that $g(m,x)$ is convex with respect to $x$; however, the infinite dimensional structure of the problem makes it difficult to prove.
	Second, with this additional assumption, it is worth noting that LL monotonicity implies displacement monotonicity and therefore is sufficient to prove uniqueness \cite{gangbo2022mean}.
	Third, the assumption that $g(m,x)$ is convex with respect to $x$ implies that, for a given $m$, each player has only one optimal decision available.
	This has an important consequence for the structure of the game: it means that if we take an initial measure that is empirical (discrete) with $n$ points, then both the game and its potential become $n$-dimensional problems.
	This can be seen as follows.
	To start with the potential, we see that minimizers of $\tilde{\s{K}}$ satisfy the first order condition
	\begin{equation}
		\frac{Y-X}{T} + D \tilde G(Y) = 0 \Leftrightarrow Y + TD_x g(\sharp(Y),Y) = X.
	\end{equation}
	Since $g(m,x)$ is assumed to be convex with respect to $x$, the function $I + tD_x g(m,\cdot)$ is invertible for every $T \geq 0$ and $m \in \sr{P}_2(\bb{R}^d)$.
	The equilibrium is thus characterized by
	\begin{equation}
		Y = \del{I + TD_x g(\sharp(Y),\cdot)}^{-1}(X).
	\end{equation}
	In particular, if $\sharp(X)$ is an empirical measure of the form $m_{\bf x}^n = \frac{1}{n}\sum_{j=1}^n \delta_{x_j}$, then $\sharp(Y) = m_{\bf y}^n$ where $y_j = \del{I + tD_x g(m_{\bf y}^n,\cdot)}^{-1}(x_j)$.
	In this case, it is sufficient to optimize over vectors in $(\bb{R}^d)^n)$, i.e.~we can replace $\tilde{\s{K}}$ with
	\begin{equation} \label{Kn}
		\tilde{\s{K}}_n({\bf y}) = \frac{1}{2Tn}\abs{{\bf x} - {\bf y}}^2 + G_n({\bf y}).
	\end{equation}
	The same considerations apply to the game itself.
	For a given vector ${\bf z}$, each individual starting at a point $x \in \bb{R}^d$ seeks to minimize the cost
	\begin{equation}
		\frac{\abs{x-y}^2}{2T} + g(m_{\bf z}^n,y),
	\end{equation}
	which implies that player will choose
	\begin{equation}
		y = \del{I + TD_x g(m_{\bf z}^n,\cdot)}^{-1}(x).
	\end{equation}
	It follows that, starting from the discrete measure $m_{\bf x}^n$, the final measure will be $m_{\bf y}^n$ with $y_j = \del{I + TD_x g(m_{\bf z}^n,\cdot)}^{-1}(x_j)$.
	A Nash equilibrium in this case is equivalent to a vector ${\bf z}$ such that ${\bf y} = {\bf z}$, i.e.~ $y_j = \del{I + TD_x g(m_{\bf y}^n,\cdot)}^{-1}(x_j)$, which is the same condition as for minimizers of the potential.

	\begin{proof}[Proof of Theorem \ref{thm:uniqueness}]
		We have already observed that (1) and (2) are equivalent  \cite{gangbo2022global}.
		
		To see that (1) implies (3), note that if $G$ is displacement convex, then $\tilde{K}$ defined in \eqref{tildeK} is strictly convex, so it has a unique minimizer.
		If $m_T$ is a minimizer of $\s{K}$, then the proof of Lemma \ref{lem:lift K} implies it is the law of $Y$, hence $m_T$ is unique.
		Similarly, by the proof of Lemma \ref{lem:straight lines}, $\s{J}$ has a unique minimizer, as desired.
		
		To show that (1) implies (4), see \cite{gangbo2022global}.
		To see that (4) implies (3), note that every minimizer of $\s{J}$ is an equilibrium, so if there is more than one minimizer there is more than one equilibrium.
		
		To see that (4) implies (1), note that if $\s{J}$ has a unique minimizer, then $\tilde{\s{K}}$ must have a unique minimizer $Y$.		
		We are going to show that $\tilde{G}$ must be convex.
		Assume to the contrary that it is not convex.
		Then there exist random variables $X,Y$ such that 
		\begin{equation} \label{eq:not convex}
			\tilde{G}(X) > \frac{1}{2}\del{\tilde{G}(X+Y) + \tilde{G}(X-Y)}.
		\end{equation}
		Any $X \in \bb{H}$ can be approximated up to arbitrary precision by a discrete random variable whose law is $m_{\bf x}^n$ for some sufficiently large $n$.
		By the continuity of $G$, \eqref{eq:not convex} implies that for some sufficiently large $n$, $G_n$ is not convex.
		It follows that for some ${\bf x} \in (\bb{R}^d)^n$ and $T > 0$, $\tilde{\s{K}}_n$ defined in \eqref{Kn} has more than one minimizer, using arguments that hold in finite dimensional settings (see e.g.~\cite[Theorem A.1]{graber2023monotonicity}; note that the linear growth of $\tilde G$ assumed here is sufficient to make this argument go through).
		This is a contradiction; therefore, $\tilde G$ must be convex, i.e.~$G$ is displacement convex.
		The proof is complete.
	\end{proof}

	Notice that the proof of Theorem \ref{thm:uniqueness} uses the assumption that $g(m,x)$ is convex with respect to $x$ only to prove that (4) implies (1).
	If this implication could be proved without restricting to finite dimensions, then the extra assumption could be eliminated.

	\subsection{Hamilton-Jacobi equation and master equation}

	We now introduce the value function
	\begin{equation}
		U(t,\mu) := \inf\cbr{\frac{1}{2t}W_2(\mu,\nu)^2 + G(\nu) : \nu \in \sr{P}_2(\bb{R}^d)}.
	\end{equation}
	The ``lifted'' version of it is
	\begin{equation} \label{value hilbert}
		\tilde{U}(t,X) := \inf\cbr{\frac{1}{2t}\enVert{X-Y}^2 + \tilde{G}(Y) : Y \in \bb{H}} = U\del{t,\sharp(X)},
	\end{equation}
	where the second equality follows from Lemma \ref{lem:lift K}.
	By classical results \cite{crandall1985hamilton,crandall1986hamilton}, $\tilde{U}$ is the unique viscosity solution of the Hamilton-Jacobi equation
	\begin{equation} \label{eq:hj H}
		\partial_t \tilde{U} + \frac{1}{2}\enVert{D_X \tilde U}^2 = 0, \quad \tilde U(0,X) = \tilde G(X).
	\end{equation}
	By \cite{gangbo2019differentiability}, $U$ itself is the unique viscosity solution (see the reference for a definition) of the corresponding PDE on Wasserstein space:
	\begin{equation} \label{eq:hj W}
		\partial_t U + \frac{1}{2}\int_{\bb{R}^d} \abs{D_\mu U(t,m,x)}^2 \dif m(x) = 0, \quad U(0,m) = G(m).
	\end{equation}
	Let us suppose for the moment that $U$ is smooth.
	Taking a linear derivative of \eqref{eq:hj W}, we deduce that $U$ has a linear derivative $u = u(t,m,x)$ satisfying
	\begin{equation} \label{master eq}
		\partial_t u(t,m,x) + \frac{1}{2}\abs{D_x u(t,m,x)}^2 + \int_{\bb{R}^d} D_y u(t,m,y)\cdot D_\mu u(t,m,x,y)\dif m(y) = 0,
	\end{equation}
	and such that $u(0,m,x)$ is a linear derivative of $G$.
	Equation \eqref{master eq} is the \emph{master equation} for the mean field game.
	It can be seen as the limit as $N \to \infty$ of the system of $N$ coupled Hamilton-Jacobi equations for an $N$-player symmetric game, or it may be seen as the decoupling equation for the system of PDE given by \eqref{mfg system}; see e.g.~\cite{cardaliaguet2019master,bensoussan2015master,bensoussan2017interpretation} for a full exposition.
	In the context of potential mean field games, it is not a mere coincidence that the master equation \eqref{master eq} is obtained from taking the linear derivative of the Hamilton-Jacobi equation \eqref{eq:hj W}.
	Indeed, $U$ is the value function for the minimization of the potential, while $u$ is the value function for an individual player; more precisely, $u(t,m,x)$ is the minimal cost for a player starting at position $x$ with time horizon $t$, where $m$ is the initial distribution of players in the game.
	Just as the cost to individuals $J$ is assumed to have a potential $\s{J}$, the individual value function $u = \min J$ in turn has a potential $U = \min \s{J}$.
	Several results on classical solutions to the master equation are based on precisely this idea: start with a smooth solution to the Hamilton-Jacobi equation on the space of measures, and then differentiate to get the master equation \cite{bensoussan2018control,bensoussan2020control,gangbo2022global}; see also \cite{gangbo2015existence}.
	
	The master equation is, both in the present context and more generally, a nonlinear transport equation on the space of measures.
	The forward-backward system of PDE, of which \eqref{mfg system} is a standard class, can be seen as the characteristics for this transport.
	The whole problem of finding a Nash equilibrium can be viewed as a problem of identifying a characteristic that intersects a given point $(t,m)$.
	If we feed the system the measure $m_0 = m$ as an initial condition and the time $T = t$ as a time horizon, the solution $u(0,x)$ of \eqref{mfg system}, which is the value function over the whole time horizon for a representative player, can be identified as the solution to the master equation for inputs $(t,m,x)$.
	In the context of potential games, one might expect that the derivative of a Hamilton-Jacobi equation should be a Burgers type equation, and the characteristics of the one should correspond to the characteristics of the other.
	Indeed, this connection can be made clearer by considering $D_x u(t,m,x) = D_\mu U(t,m,x)$, which formally satisfies $D_x u(t,\sharp(X),X) = \tilde{\bf u}(t,X)$ where
	\begin{equation} \label{vector master}
		\partial_t \tilde{\bf u}(t,X) + \bb{E}\sbr{\tilde{\bf u}(t,X)\cdot D_X \tilde{\bf u}(t,X)} = 0,
		\quad \tilde{\bf u}(0,X) = D\tilde{G}(X).
	\end{equation}
	(Cf.~\cite{delarue2019restoring}.)
	This is often referred to as the \emph{vectorial master equation}.
	
	When we assume that $g(m,x)$ is convex with respect to $x$, we can make a more direct connection with classical theory.
	In this case define the restriction to empirical measures:
	\begin{equation}
		U_n(t,{\bf x}) = U(t,m_{\bf x}^n) = \inf\cbr{\frac{1}{2tn}\abs{{\bf x} - {\bf y}}^2 + G_n({\bf y})}.
	\end{equation}
	Then $U_n$ is the unique viscosity solution \cite{crandall1983viscosity,crandall1992user} to the Hamilton-Jacobi equation
	\begin{equation} \label{hj fin dim}
		\partial_t U_n + \frac{n}{2}\abs{D_{\bf x} U_n}^2 = 0, \quad U_n(0,{\bf x}) = G_n({\bf x}).
	\end{equation}
	By Krushkov's results \cite{kruvzkov1967generalized}, the gradient ${\bf u}_n(t,{\bf x}) = D_{\bf x} U_n$ is the unique entropy solution of the hyperbolic system
	\begin{equation} \label{burgers}
		\partial_t {\bf u}_n + n\del{{\bf u}_n \cdot D_{\bf x}}{\bf u}_n = {\bf 0}, \quad  {\bf u}_n(0,{\bf x}) = \frac{1}{n}D_x g(m_{\bf x}^n,{\bf x}),
	\end{equation}
	where we have used \eqref{eq:finite derivative}.
	Indeed, the well-posedness of \eqref{burgers} is a corollary of the well-posedness of \eqref{hj fin dim}; much less is known about hyperbolic systems that are not the gradient of a Hamilton-Jacobi equation.
	
	\subsection{Selection}
	
	We have seen that a potential game could very well have multiple Nash equilibria (Theorem \ref{thm:uniqueness}).
	Indeed, one need only take particular instances of \eqref{burgers} for which it is well-known that characteristics cross, for instance by taking $d = n = 1$ and assuming $D_x g(\delta_x,x)$ is a decreasing function of $x$.
	The problem of selecting a particular Nash equilibrium for a mean field game has been a subject of active research in recent years.
	First of all, Delarue has shown that uniqueness in mean field games can be restored by adding a common noise \cite{delarue2019restoring}.
	This suggests a strategy akin to Selten's idea in defining trembling hand perfect equilibrium: if we add a common noise to the game and allow the intensity to converge to zero, this will select one of multiple equilibria.	
	In this vein, Delarue and Tchuendom studied a linear-quadratic mean field game, and showed that the Burgers equation, not coincidentally played a key role in selection \cite{delarue2020selection}.
	Cecchin and Delarue showed that, for a finite state mean field game, the vanishing noise limit would select the entropy solution to the master equation \cite{cecchin2022selection}.
	Presumably motivated by these results, Cecchin and Delarue have also introduced a notion of weak solution for potential mean field games with continuous state space \cite{cecchin2022weak}, and it is precisely the kind of ``entropy solution'' one should expect based on the following analogy.
	The derivative of the unique viscosity solution of a (classical) Hamilton-Jacobi equation is the unique entropy solution of the hyperbolic system obtained by taking the gradient of the Hamilton-Jacobi equation; the master equation is the derivative of a Hamilton-Jacobi equation on the space of measures; therefore, the derivative of the solution to the Hamilton-Jacobi equation on the space of measures should be the unique entropy solution to the master equation.
	Using a highly technical approach, Cecchin and Delarue introduce a measure $\bb{P}$ on $\sr{P}_2(\bb{T}^d)$ (where $\bb{T}^d$ is the torus in $d$ dimensions) such that uniqueness of the entropy solution holds in a $\bb{P}$-a.e.~sense.
	Their approach involves approximating measures using Fourier coefficients, which is suitable in the case where diffusion terms appear in the dynamics so that the continuity equation becomes a second-order PDE with a Laplacian.
	Establishing a general uniqueness theorem for entropy solutions in the first-order case seems to be an open problem.
	
	Let us now clarify why the concept of entropy solution is a natural choice for a selection principle.
	The characteristics for the (vectorial) master equation \eqref{vector master} are curves $X(s)$ satisfying
	\begin{equation} \label{characteristics}
		\dot{X}(s) = \tilde{\bf u}(s,X(s)), \quad \od{}{s}\tilde{\bf u}(s,X(s)) = 0,
	\end{equation}
	hence $\tilde{\bf u}(s,X(s)) = D\tilde G(X(0))$ for all $s$.
	Setting $Y = X(0)$, we deduce that a characteristic hitting $X$ at time $t$ is given by $X(s) = Y + sD\tilde{G}(Y)$ where $Y$ solves $Y + tD\tilde{G}(Y) = X$.
	Rewriting this condition as $Y + tD_xg(\sharp(Y),Y) = X$, we see it is exactly the same as the definition of Nash equilibrium (i.e.~given that $Y$ represents the anticipated final distribution, $Y$ is the optimal response from starting position $X$).
	We see that the Nash equilibrium is unique if and only if characteristics do not cross.
	The classical theory of entropy solutions was developed precisely to deal with the case when characteristics do cross; see e.g.~ \cite{kruvzkov1967generalized,kruvzkov1970first,dafermos1983hyperbolic,lax1973hyperbolic,smoller2012shock}.
	Of course, there are two key differences between the present question of finding Nash equilibria and the classical problem of studying nonlinear fluid flow.
	The first difference is that our present setting is infinite dimensional.
	This can be partially alleviated if we take the assumption that $g(m,x)$ is convex with respect to $x$, and we take our initial measure to be discrete.
	Then the equilibrium corresponds to finding characteristics in $(\bb{R}^d)^n$ of the classical conservation law \eqref{burgers}.
	
	The second difference is apparently much more fundamental.
	In the classical theory of conservation laws, we are interested in fluid flow, and therefore our criterion for selecting a solution to the PDE should be motivated by physics.
	Indeed, the name ``entropy solution'' comes from the principle that physical entropy must increase with time.
	On the other hand, if we are trying to choose a Nash equilibrium, it is far from clear whether the notion of increasing entropy is meaningful.
	Nevertheless, we there is at least one other physical interpretation that could be useful: by \cite{kruvzkov1967generalized}, the unique entropy solution to \eqref{burgers} is the limit as $\varepsilon \to 0$ of solutions to the parabolic equations
	\begin{equation} \label{burgers1}
		\partial_t {\bf u}_n + n\del{{\bf u}_n \cdot D_{\bf x}}{\bf u}_n = \varepsilon \Delta{\bf u}_n, \quad  {\bf u}_n(0,{\bf x}) = \frac{1}{n}D_x g(m_{\bf x}^n,{\bf x}).
	\end{equation}
	The Laplacian term $\varepsilon \Delta{\bf u}_n$ comes from diffusion, which in physics arise from viscosity of a fluid but in a game can arise from uncertainty or ``noise.''
	The interpretation is that players cannot be certain the crowd will evolve according to the characteristic ${\bf x}(s)$ (or, more properly speaking, $m_{{\bf x}(s)}^n$), so they assume that ${\bf x}(s)$ is perturbed by a Brownian motion with small variance and that the value ${\bf u}_n(s,{\bf x}(s))$ is a martingale.
	The perturbation is thus a ``common noise,'' which affects the trajectory of the whole crowd, to be distinguished from idiosyncratic noise, which affects each player independently \cite{carmona2016mean}.
	The entropy solution of \eqref{burgers} thus gives us the \emph{vanishing common noise limit equilibrium}, provided the common noise is Brownian motion.
	
	Here we see the basic idea of Selten \cite{selten1975,harsanyi1988general} coming through quite clearly.
	Recall that when the set of strategies is finite, players' uncertainty can be expressed by assigning a uniform small probability to every alternative strategy, and the result will be ``trembling hand perfect equilibrium.''
	Since the strategy set in this game is infinite, it is not possible to perturb the set of actions by a uniform distribution, but it is possible to choose a normal distribution (via Brownian motion), and this, too, will result in a refinement of Nash equilibrium (the vanishing common noise limit).
	As for the fact that we have restricted to finite dimensions, this can be interpreted in the following way.
	Given an arbitrary initial distribution, players will all take the same discrete approximation of it.
	They will then play the game as if that discrete approximation were the true distribution.
	In principle, the final discrete distribution should be a good approximation of the true equilibrium.
	
	Notice that the vanishing common noise limit equilibrium, which is the entropy solution, is also given by the gradient of the value function for the potential.
	In fact, we can use the potential to designate a selection principle without restricting to finite dimensions.
	Observe that for $t > 0$, the value function $\tilde U$ given in \eqref{value hilbert} is Lipschitz and semi-concave, i.e.~there exists a constant $C$ such that
	\begin{equation}
		\lambda \tilde U(t,X) + (1-\lambda) \tilde U(t,Z) - \tilde U(t,\lambda X + (1-\lambda)Z) \leq \frac{C}{t}\lambda(1-\lambda)\enVert{X-Z}^2
	\end{equation}
	for all $\lambda \in (0,1)$ and $X,Z \in \bb{H}$.
	(In this case we can take $C = 1$.)
	By a result of Albano and Cannarsa \cite{albano1999singularities}, the set of points $\Sigma$ on which $\tilde U(t,X)$ is not differentiable is countably $\infty - 1$ rectifiable, meaning that it is contained in the countable union of sets, each of which is the image under a Lipschitz map of a bounded subset of a subspace with codimension 1.
	In other words, the set of $X$ for which $D_X \tilde U(t,X)$ exists is ``very large'' in a certain sense, similar to the structure of BV functions \cite{evans92}.
	For such $X$, we see that $Y = X - tD_X \tilde U(t,X)$ is the unique minimizer of the potential; this follows from standard arguments showing that the optimal control is unique at points where the value function is differentiable \cite{cannarsa2004semiconcave}.
	In the present context, we can see this by observing that if $Y$ is any minimizer then
	\begin{equation}
		\tilde U(t,X+Z) - \tilde U(t,X) \leq \frac{\enVert{X+Z-Y}^2}{2t} + \tilde G(Y) - \frac{\enVert{X-Y}^2}{2t} - \tilde G(Y) = \ip{\frac{X-Y}{t}}{Z} + \frac{1}{2t}\enVert{Z}^2
	\end{equation}
	for all $Z$, hence $\frac{X-Y}{t} = D_X \tilde{U}(t,X)$.
	
	Note that even when the minimizer of the potential is unique, the equilibrium may not be unique.
	A simple but illustrative case is when $G(m) = \int \phi \dif m$ for some smooth function $\phi$ such that $\phi'(x) = 1$ when $x \leq -1$, $\phi'(x) = 0$ when $x \geq 0$, and $\phi'(x) = -x$ when $-1 \leq x \leq 0$.
	Take $d = 1$ and take the initial distribution to be a Dirac mass $\delta_x$.
	The Nash equilibrium condition is to find a $y$ such that $y + t\phi'(y) = x$, while the potential to be minimized is $\frac{\del{x-y}^2}{2t} + \phi(y)$.
	One checks that, for $t > 1$ (the time when the shock occurs) and for $0 < x < t$, there are three equilibria given by $y = x$, $y = x-t$, and $y = x/(1-t)$, but the only minimizer is $y = x$ when $x > (t-1)/2$ (to the right of the shock) and $y = x-t$ when $x < (t-1)/2$ (to the left of the shock).
	This corresponds precisely to solving the Burgers equation with a shock.
	
	Since the minimizer $Y = X + tD_X\tilde U(t,X)$ selects only one of possibly many equilibria, we can say that minimizing the potential is itself a sort of selection principle.
	It is unclear whether it can be interpreted as a vanishing noise limit; the difficulty is both in defining the noise on an infinite dimensional space (but see \cite{delarue2019restoring}) and in proving rigorously that the limit has the desired meaning.
	If we assume that $g(m,x)$ is convex with respect to $x$, then for $\sharp(X) = m_{\bf x}^n$ we have $D_X \tilde U(t,X) = {\bf u}_n(t,{\bf x})$, which is exists for a.e.~${\bf x} \in (\bb{R}^d)^n$ (in the sense of Lebesgue measure) and agrees with the entropy solution of \eqref{burgers}.
	For this reason a sort of vanishing noise limit interpretation is plausible, at least in some cases.
	However, the general case seems to be quite open.
	
	Even if we can resolve the infinite dimensionality of the problem, there are more general difficulties in justifying the selection principle proposed here.
	One difficulty, which is highlighted by the work of Lions and Seeger \cite{lions2023linear}, is that there are infinitely many types of vanishing noise one can add to characteristics, each resulting in a different selection.
	Indeed, given \emph{any} curve $c(t)$ such that $c'(t) \in (0,1)$ for all $t \geq 0$, there is a function $\theta_c(t)$ such that the limit of solutions to the PDE
	\begin{equation}
		\partial_t u + u\partial_x u = \varepsilon\del{\partial_{xx}^2 u + \theta_c(t)\abs{\partial_x u}^2}, \quad u(0,x) = \begin{cases}
			1, \quad x < 0,\\
			0, \quad x > 0
		\end{cases}
	\end{equation}
	is the piecewise defined ``solution'' to the Burgers equation given by
	\begin{equation}
		u(t,x) = \begin{cases}
			1, \quad x < c(t),\\
			0, \quad x > c(t).
		\end{cases}
	\end{equation}
	See \cite[Proposition 5.4 and Theorem 5.4]{lions2023linear}.
	The meaning of this result is that the region in which characteristics cross can be divided into two regions using a more or less arbitrary ``shock'' curve (which in general does not at all correspond to the physical entropy criterion), and the selected solution can still be obtained by a vanishing noise limit.
	Therefore, the phrase ``vanishing noise limit'' does not uniquely specify a selection criterion.
	
	Two possible responses may be given in favor of the entropy solution.
	First, for any other vanishing noise limit, the perturbed characteristics are not martingales.
	If the notion of martingale is interpreted as an ``unbiased'' stochastic perturbation, then that could be seen as an argument in favor of the entropy solution.
	On the other hand, if there is reason to suspect that the noise will be biased in one direction or another, then the entropy solution no longer provides the appropriate selection principle.
	The second possible response is that, as we have seen, the entropy solution also corresponds to choosing minimizers of the potential.
	If there is a plausible evolutionary mechanism by which the population converges only to minimizers of the potential (cf.~\cite{cardaliaguet2017learning,sandholm2001potential,monderer1996fictitious}), then it is natural to use the entropy solution to select the appropriate Nash equilibrium.
	On the other hand, this justification is contingent on finding and arguing for such an evolutionary mechanism.
	
	Other problems using entropy solutions as a selection principle can arise when the mean field game is \emph{not} potential.
	See \cite{graber2024some} for a discussion of some examples.
	
	\section{Conclusion}
	
	In this article, we have briefly surveyed the concept of a potential mean field game and connected it to what was already known about potential games.
	We have only discussed the first-order case.
	Second-order problems may be more or less challenging, depending on the structure of the second-order terms and how they affect the dynamics of the game.
	For instance, adding an idiosyncratic noise that affects each individual independently could make equations in \eqref{mfg system} strictly parabolic, in which case a number of PDE techniques can be brought to bear on the analysis \cite{porretta2015weak,gomes2015time,gomes2015time2,gomes2016time,cardaliaguet2012long}.
	On the other hand, the addition of a common Brownian motion, affecting all individuals simultaneously, does not appear to add additional smoothness to solutions of the standard master equation \cite{cardaliaguet2019master}.
	Insofar as we have dealt with the idea of a common noise in this article, it has only been in the case where the game is essentially finite dimensional; the infinite dimensional analog is not trivial to define \cite{delarue2019restoring}.
	
	In any case, restricting to the first-order case has allowed us, first of all, to provide a transparent proof that minimizers of the potential are also Nash equilibria, defined according to a Lagrangian framework.
	This can be done even in state-constrained problems.
	Second, we have been able to make a clear connection between solving a mean field game and analyzing characteristics of a nonlinear hyperbolic balance law.
	This has the benefit of motivating potential approaches to the selection of equilibria, which is a major open problem in mean field game theory.
	The recent breakthrough by Cecchin and Delarue \cite{cecchin2022weak} seems to confirm that potential mean field games are the natural place to start a rigorous analysis of selection principles.
	It is to be hoped that the present article helps clarify the main issues, and that it may serve as a catalyst for new developments.
	
	\medskip
	
	\noindent {\bf Conflict of interest.} The author declares that he does not have any conflicts of interests.	
	
	\medskip
	
	\noindent {\bf Data Availability Statement.} Data sharing not applicable to this article as no datasets were generated or analyzed during the current study.
	
	\bibliographystyle{siam}
	\bibliography{../../mybib/mybib}
\end{document}